\documentclass[10pt,a4paper,reqno]{amsart}

\usepackage[utf8]{inputenc}
\usepackage[T1]{fontenc}
\usepackage{xspace,
	amsthm,
	amsmath,
	amssymb,
	bbm,
	tikz,
	graphicx,
	subfig,
	xcolor,
	keyval}
\DeclareGraphicsExtensions{.png, .pdf}

\usepackage[hidelinks]{hyperref}

\usepackage{geometry}
\tikzstyle{nodino}=[circle,draw,fill,inner sep=0pt,minimum size=0.5mm]
\tikzstyle{infinito}=[circle,inner sep=0pt,minimum size=0mm]
\tikzstyle{nodo}=[circle,draw,fill,inner sep=0pt, minimum size=0.5*width("k")]
\tikzstyle{nodo_vuoto}=[circle,draw,inner sep=0pt, minimum size=0.5*width("k")]
\usetikzlibrary{graphs}
\tikzset{every loop/.style={min distance=10mm,in=300,out=240,looseness=10}}
\tikzset{place/.style={circle,thick,draw=blue!75,fill=blue!20,minimum
		size=6mm}}
\tikzset{place2/.style={circle,thick,draw=red!75,fill=red!20,minimum
		size=6mm}}

\newcommand{\rr}{{\mathbb R}}

\newcommand{\G}{{\mathcal{G}}}
\newcommand{\udot}{\|u'\|_{2,\mathcal{G}}}
\newcommand{\uLp}{\|u\|_{p,\mathcal{G}}}
\newcommand{\uLtwo}{\|u\|_{2,\mathcal{G}}}
\newcommand{\HmuG}{H_\mu^1(\mathcal{G})}
\newcommand{\uLsix}{\|u\|_{6,\mathcal{G}}}

\newcommand{\uLsixcompact}{\|u\|_{6,\K}}

\newcommand{\K}{\mathcal{K}}
\newcommand{\h}{\mathcal{H}}

\newcommand{\EEK}{\mathcal{E}_\G(\mu,\K)}
\newcommand{\dx}{\,dx}
\newcommand{\ep}{\varepsilon}
\newcommand{\vv}{\mathrm{v}}
\newcommand{\f}[2]{\frac{#1}{#2}}
\newcommand{\tf}[2]{\tfrac{#1}{#2}}
\newcommand{\wt}[1]{\widetilde{#1}}

\theoremstyle{plain} 
\newtheorem{thm}{Theorem}[section] 
 
\newtheorem{lem}[thm]{Lemma} 
\newtheorem{prop}[thm]{Proposition} 

\theoremstyle{definition}

\theoremstyle{definition}

\theoremstyle{remark} 
\newtheorem{rem}{Remark}[section]

\begin{document}

\title[$L^2$-critical NLS on graphs with localized nonlinearity]{$L^2$-critical NLS on noncompact metric graphs with localized nonlinearity: topological and metric features}

\author[S. Dovetta]{Simone Dovetta}
\address{Politecnico di Torino, Dipartimento di Scienze Matematiche ``G.L. Lagrange'', Corso Duca degli Abruzzi, 24, 10129, Torino, Italy}
\address{Universit\`a degli Studi di Torino, Dipartimento di Matematica ``G. Peano'', Via Carlo Alberto, 10, 10123, Torino, Italy}
\email{simone.dovetta@polito.it}
\author[L. Tentarelli]{Lorenzo Tentarelli}
\address{Universit\`{a} degli Studi di Roma``La Sapienza'', Dipartimento di Matematica ``G. Castelnuovo'', P.le Aldo Moro, 5, 00185, Roma, Italy.}
\email{tentarelli@mat.uniroma1.it}

\date{\today}

\begin{abstract} 
 Carrying on the discussion initiated in \cite{DT-p}, we investigate the existence of ground states of prescribed mass for the $L^2$-critical NonLinear Schr\"odinger Equation (NLSE) on noncompact metric graphs with localized nonlinearity. Precisely, we show that the existence (or nonexistence) of ground states mainly depends on a parameter called reduced critical mass, and then we discuss how the topological and metric features of the graphs affect such a parameter, establishing some relevant differences with respect to the case of the extended nonlinearity studied by \cite{AST-CMP}. Our results rely on a thorough analysis of the optimal constant of a suitable variant of the $L^2$-critical Gagliardo-Nirenberg inequality.
\end{abstract}

\maketitle
	
	\vspace{-.5cm}
    \smallskip

    
	
	\section{Introduction}
	
	In this paper, we aim at discussing the existence of ground states for the NonLinear Schr\"odinger Equation (NLSE) on metric graphs with localized nonlinearities.
	
	We briefly recall that a metric graph is the locally compact metric space which arises as one endows a \emph{multigraph} $\G=(\mathrm{V},\mathrm{E})$ with a parametrization that associates each bounded edge $e\in\mathrm{E}$ with a closed and bounded interval $I_e=[0,\ell_e]$ of the real line, and each unbounded edge $e\in\mathrm{E}$ with a half-line $I_e=\rr^+$ (for details see \cite{AST-CVPDE,BK} and references therein). Consistently, functions on metric graphs $u=(u_e)_{e\in\mathrm{E}}:\G\to\rr$ are families of functions defined on each edge $u_e:I_e\to\rr$ in such a way that $u_{|_e}=u_e$. Lebesgue spaces are, then, given by
	\[
	 L^p(\G):=\bigoplus_{e\in\mathrm{E}}L^p(I_e),\qquad p\in[1,\infty],
	\]
	while
	\[
	 H^1(\G):=\bigg\{u\in\bigoplus_{e\in\mathrm{E}}H^1(I_e)\,:\,u \text{ is continuous on }\G\bigg\},
    \]
    both equipped with the natural norms, denoted by $\|u\|_{p,\G}$ and $\|u\|$ respectively.  In addition, in the following we limit ourselves to the study of those graphs $\G$ such that
	\begin{itemize}
	 \item[\textbf{(A)}] $\G$ is \emph{connected}, \emph{noncompact},  with a \emph{finite number of edges} and with \emph{non-empty compact core} $\K$ (which is the subgraph of the bounded edges of $\G$).
	\end{itemize}

	\medskip
    In view of this, the central issue discussed in the present paper is that of the existence of minimizers for the $L^2$-\emph{critical} NLS energy functional 
	\begin{equation}
		\label{EQ-energy def}
		E(u,\K):=\frac{1}{2}\int_\G|u'|^2\,dx-\frac{1}{6}\int_\K|u|^6\,dx
	\end{equation}
	under the \emph{mass constraint}
	\[
		\|u\|_{2,\G}^2=\mu>0\,,
	\]
	that is, the so-called \emph{ground states}. In other words, we discuss the existence of functions $u\in\HmuG$ such that
	\begin{equation}
		\label{EQ-minimum problem def}
		E(u,\K)=\mathcal{E}_\G(\mu,\K):=\inf_{v\in\HmuG}E(v,\K),
	\end{equation}
    with    
    \[
    	\HmuG=\big\{u\in H^1(\G)\,:\,\|u\|_{2,\G}^2=\mu\big\}\,.
    \]
    Such functions are known (the proof is completely analogous to that of \cite[Proposition 3.3]{AST-CVPDE}) to be $L^2$-solutions of
    \[
     -\Delta_\G u+\lambda u=\chi_\K\,|u|^{4}\,u,\qquad\lambda\in\rr^+,
    \]
    where $\chi_\K$ is the characteristic function of $\K$ and $-\Delta_\G$ is the so-called \emph{Kirchhoff} Laplacian, namely the operator defined by
    \begin{gather}
     -\Delta_\G u_{|\mathring{I}_e}:=-u_e''\nonumber\\[.2cm]
     \label{eq-domkirc}\mathrm{dom}(-\Delta_\G):=\bigg\{u\in H^1(\G):u_e\in H^2(I_e),\:\forall e\in\mathrm{E},\:\text{and}\:\sum_{e\succ v}\tf{du_e}{dx_e}(\vv)=0,\:\forall\vv\in\K\bigg\}
    \end{gather}
    with $e\succ v$ meaning that the edge $e$ is incident at the vertex $\vv$ and $\tf{du_e}{dx_e}(\vv)$ standing for $u_e'(0)$ or $-u_e'(\ell_e)$ according to whether $x_e$ is equal to $0$ or $\ell_e$ at $\vv$. It is, also, straightforward that $\psi(t,x):=e^{\imath\lambda t}u(x)$, with $u$ ground state, is a \emph{stationary solution} of the time-dependent nonlinear Schr\"odinger equation
    \begin{equation}
     \label{eq-timeNLS}
     \imath\f{\partial\psi}{\partial t}=-\Delta_\G\psi-\chi_\K\,|\psi|^4\,\psi.
    \end{equation}
	
	\medskip
	Problem \eqref{EQ-minimum problem def} was first proposed in \cite{DT-p}, in the ``toy'' model of the \emph{tadpole} graph (see Figure \ref{fig-tadpole}). Here we aim at extending that seminal result in order to present an (almost) complete classification of the whole phenomenology.
	
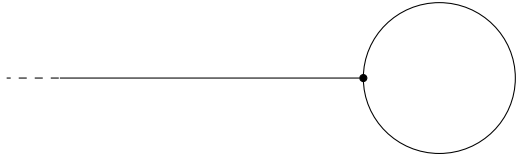
\begin{figure}
 \centering
\begin{tikzpicture}
\node at (-1,0) [nodo] (-10) {};
\node at (-5,0) [infinito] (-50) {};
\node at (-5.7,0) [infinito]  (-60) {};
\draw (0,0) circle (1cm);
\draw [-] (-10) -- (-50);
\draw [dashed] (-50) -- (-60);
\end{tikzpicture}
\caption{\emph{tadpole} graph.}
\label{fig-tadpole}
\end{figure}
	
	 The study of the NLSE on graphs has recently become a quite popular research topic and then the literature has hugely increased. We mention, e.g., \cite{ADST,AST-CVPDE,AST-JFA,CFN-N,KP-JDE,LLS-JMAA} (and the references therein) for problems involving the nonlinearity extended to the whole graph in the $L^2$-subcritical case, and \cite{CDS-p,D-JDE,D2-p,MP-AMRX} for the discussion of the Schr\"odinger equation on compact graphs. We also mention two recent works on some other dispersive equations: \cite{MNS-APDE} for the \emph{KdV equation} and \cite{BCT-p} for the \emph{Nonlinear Dirac equation}.
	
	Furthermore, with respect to the present paper, it is worth recalling \cite{AST-CMP}, which deals with the existence of $L^2$-critical ground states in the extended case, where \eqref{EQ-energy def} is replaced by
	\begin{equation}
	 \label{eq-en_extended}
	 E(u):=\frac{1}{2}\int_\G|u'|^2\,dx-\frac{1}{6}\int_\G|u|^6\,dx,
	\end{equation}
	 and \cite{ST-JDE,ST-NA,T-JMAA}, which deals with the localized $L^2$-\emph{subcritical} problem, where \eqref{EQ-energy def} is replaced by 
	\[
		E(u,\K,p):=\frac{1}{2}\int_\G|u'|^2\,dx-\frac{1}{p}\int_\K|u|^p\,dx,\qquad p\in(2,6)\,.
	\]
	
\medskip
Before stating our main theorems, let us mention some classical results on the $L^2$-critical issue with extended nonlinearity on the real line and the half-line (see \cite{C}). Precisely, it is well known that, letting
	 \[
	 \mathcal{E}_\G(\mu):=\inf_{u\in\HmuG}E(u),
	 \]
	 with $E(\cdot)$ defined by \eqref{eq-en_extended} and $\G=\rr$, then a threshold phenomenon shows up, that is
	 \begin{equation}
	 \label{EQ-inf rr}
	 	\mathcal{E}_\rr(\mu)=\begin{cases}
	 	0 & \text{if }\mu\leq\mu_\rr\\
	 	-\infty & \text{if }\mu>\mu_\rr\,,
	 	\end{cases}
	 \end{equation}
    where $\mu_\rr=\pi\sqrt{3}/2$ and is usually called \emph{critical mass} of the real line. Moreover, $\mathcal{E}_\rr(\mu)$ is attained if and only if $\mu=\mu_\rr$, and a whole family of ground states exists, the \textit{solitons} $\{\phi_\lambda\}_{\lambda>0}$, given by
	 \begin{equation}
	 	\label{EQ-def soliton lambda}
	 	\phi_\lambda(x):=\sqrt{\lambda}\phi_1(\lambda x),\qquad\phi_1(x):=\text{sech}^{1/2}(2x/\sqrt{3})\,,
	 \end{equation}
	 and satisfying $E(\phi_\lambda,\rr)=0$, for every $\lambda>0$.
	 
	 Analogously, when $\G=\rr^+$, the portrait is the same with $\mu_\rr$ replaced by the critical mass of the half-line $\mu_{\rr^+}=\mu_\rr/2=\sqrt{3}\pi/4$. Furthermore,  ground states exist (again) only at the critical mass and they are the so-called \textit{half-solitons}, i.e. the restrictions of $\phi_\lambda$ to $\rr^+$.
	 
	 \medskip
	 We can now state the first result of the paper, which provides a topological classification of metric graphs.
	 
	 In the following, recall that a metric graph $\G$ is said to admit a \emph{cycle covering} if and only if each edge of $\G$ belongs to a cycle,  where a cycle can be either a loop, that is a closed path of consecutive bounded edges, or an unbounded path joining the endpoints of two distinct half-lines, which are then identified as a single vertex at infinity (see \cite{AST-JFA,AST-CMP} for further details). Finally, by a {\em terminal edge} we mean any edge ending with a vertex of degree 1.

    \begin{thm}
    	\label{THM1}
    	Let $\G$ satisfy \textbf{(A)}. Then, there exists $\mu_\K\in[\mu_{\rr^+},\mu_\rr]$ such that
    	
    	\begin{equation}
    		\label{EQ- THM1 inf}
    		\mathcal{E}_\G(\mu,\K)\begin{cases}
    		=0 & \text{if }\mu\leq\mu_\K\\
    		<0 & \text{if }\mu\in(\mu_\K,\mu_\rr]\\
    		=-\infty & \text{if }\mu>\mu_\rr\,.
    		\end{cases}
    	\end{equation}
    	Moreover,
    	\begin{itemize}
    		\item[(i)] if $\G$ has at least one terminal edge (see, e.g., Figure \ref{FIGURE-theorem}(a)), then
    		\[
    		 \mu_\K=\mu_{\rr^+},\qquad \EEK=-\infty\quad\text{for all }\mu>\mu_\K,
    		\]
            and ground states never exist;
    		\item[(ii)] if $\G$ admits a cycle-covering (see, e.g., Figure \ref{FIGURE-theorem}(b)), then
    		\[
    		 \mu_\K=\mu_\rr
    		\]
    		and ground states never exist;
    		\item[(iii)] if $\G$ has exactly one half-line and no terminal edges (see, e.g., Figure \ref{FIGURE-theorem}(c)), then 
    		\begin{equation}
    		\label{EQ-bound mu tilde 1 half-line}
    			\mu_{\rr^+}<\mu_\K<\sqrt{3}
    		\end{equation}
    		and ground states of mass $\mu$ exist if and only if $\mu\in[\mu_\K,\mu_\rr]$.
    		\item[(iv)] if $\G$ has neither a terminal edge, nor a cycle-covering, and at least two half-lines (see, e.g., Figure \ref{FIGURE-theorem}(d)), then
    		\begin{equation}
    		\label{EQ-bound mu tilde resto del mondo}
    			\mu_{\rr^+}<\mu_\K\leq\mu_\rr
    		\end{equation}
    		and ground states of mass $\mu$ exist if and only if $\mu\in[\mu_\K,\mu_\rr]$, provided that $\mu_\K\neq\mu_\rr$.
    	\end{itemize}
    \end{thm}
    
	\begin{figure}[t]
		\centering
		\subfloat[][a graph with a terminal edge]
		{\includegraphics[width=.45\columnwidth]{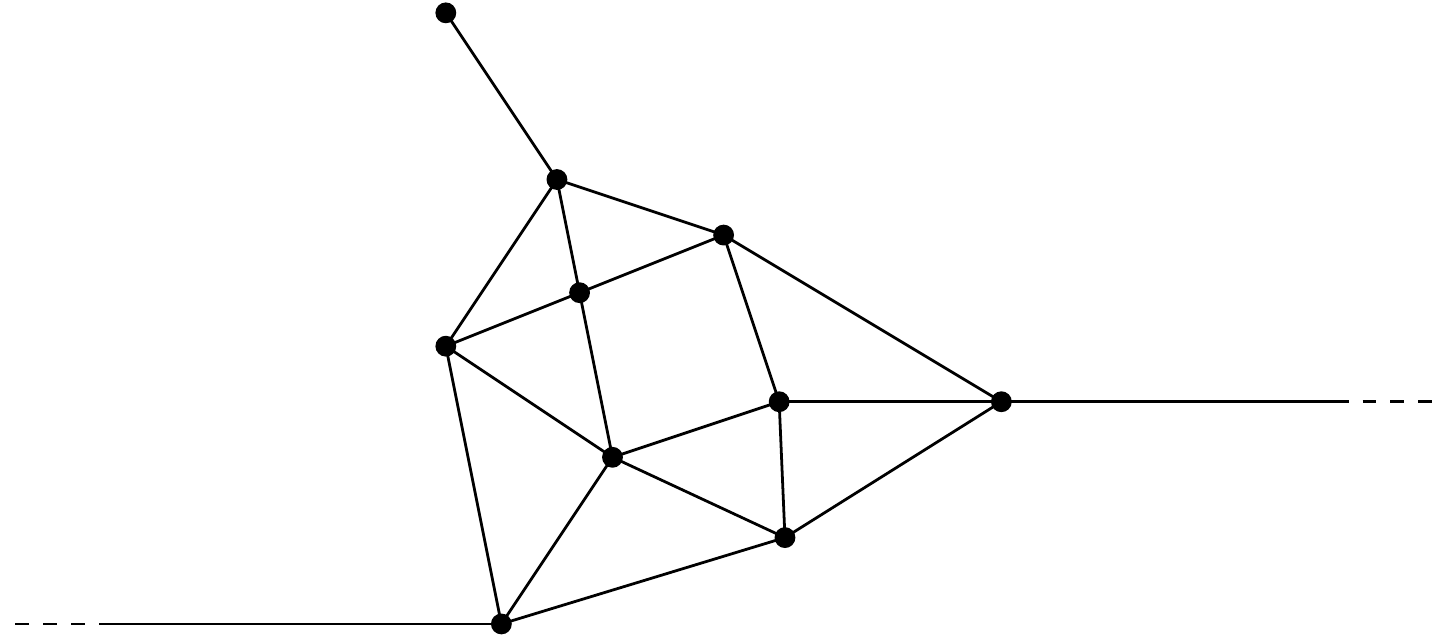}} \qquad\qquad
		\subfloat[][a graph with a cycle-covering]
		{\includegraphics[width=.45\columnwidth]{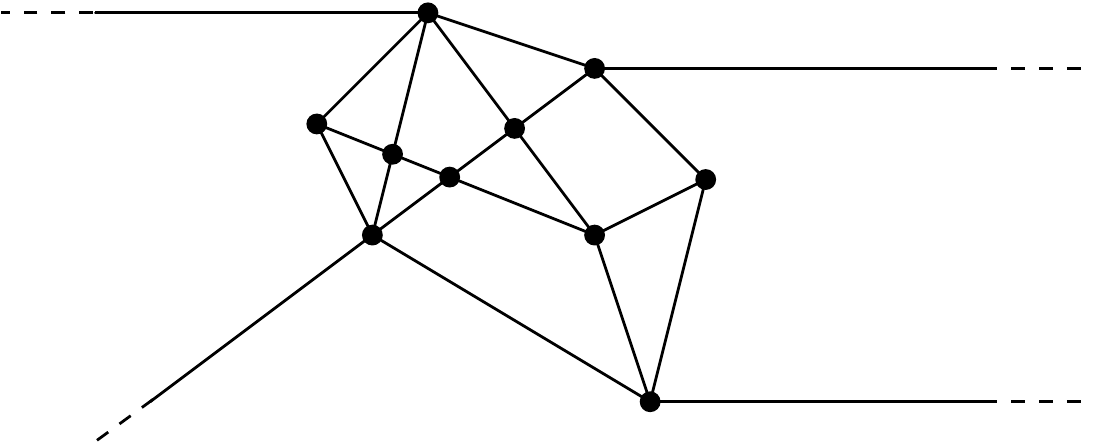}}\qquad
		
		\medskip
		\medskip
		
		\subfloat[][a graph with one half-line and no terminal edge]
		{\includegraphics[width=.45\textwidth]{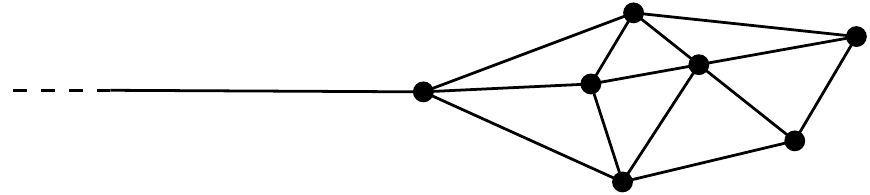}}\qquad\qquad
		\subfloat[][a graph with no terminal edge, nor cycle-covering and at least two half-lines]
		{\includegraphics[width=.45\columnwidth]{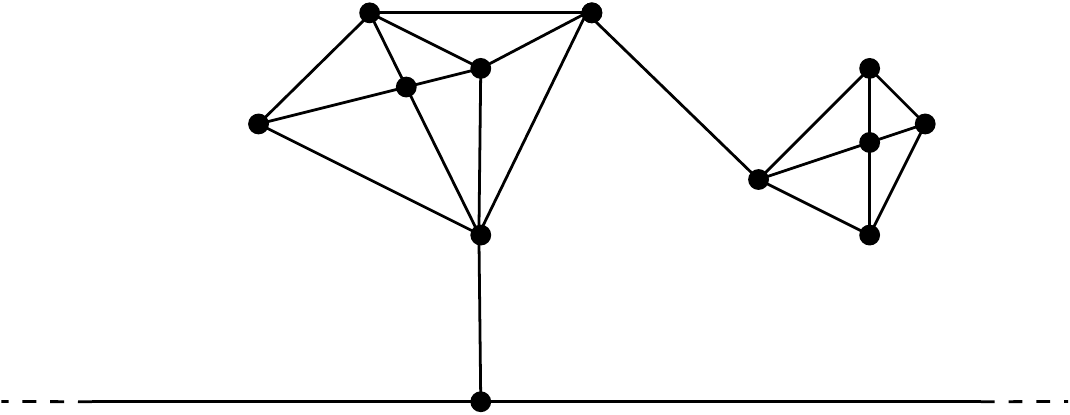}}
		\caption{examples of cases \textit{(i)-(iv)} in Theorem \ref{THM1}.}
		\label{FIGURE-theorem}
	\end{figure}

	 Preliminarily, let us point out that assumption $\mu_\K\neq\mu_\rr$ in case (iv) is consistent, in the sense that one can easily exhibit examples of graphs fulfilling it (see, for instance, the \emph{signpost} graph of Figure \ref{FIGURE-signpost}, when the vertical edge is large enough, as in Proposition \ref{PROP-METRIC SIGNPOST 1}).

	It turns out that the actual value of $\mu_\K$, that we refer to as the \emph{reduced critical mass} of $\G$ in the following, is strictly related to a Gagliardo-Nirenberg-type inequality, i.e.
	\begin{equation}
	 \label{eq-red_mass}
     \mu_\K:=\sqrt{\frac{3}{C_{\K}}}\,,
    \end{equation}
    where $C_\K$ denotes the sharpest constant of
	\begin{equation}
	 \label{eq-red_GN}
    	\uLsixcompact^6\leq C_{\K}\uLtwo^4\udot^2\,,\qquad\forall u\in H^1(\G)\,,
    \end{equation}
    namely
    \begin{equation}
    	\label{EQ-def C_6.K}
    	C_{\K}:=\sup_{u\in H^1(\G)}Q(u)\,,\qquad\text{where}\qquad Q(u):=\frac{\uLsixcompact^6}{\uLtwo^4\udot^2}\,.
    \end{equation}
	
	The dependence of the existence of ground states on a critical mass is a common feature with the issue of the extended nonlinearity discussed in \cite{AST-CMP}. However, some major differences appear. First, in \cite{AST-CMP} $\mu_\K$ is replaced by
	 \[
	  \mu_\G:=\sqrt{\frac{3}{C_\G}}\,,
	 \]
	 $C_\G$ being the optimal constant of the standard Gagliardo-Nirenberg inequality
	 \begin{equation}
	 	\label{EQ-GN 6}
	 	\|u\|_{6,\G}^6\leq C_\G\|u\|_{2,\G}^4\|u'\|_{2,\G}^2\,,\qquad\forall u\in H^1(\G)\,.
	 \end{equation}
	 Clearly,
	 \[
	 	C_{\K}\leq C_\G\,,
	 \]
	 so that, by definition,
	 \[
	 	\mu_\G\leq\mu_\K\,.
	 \]
	 
	 Furthermore, even though \cite{AST-CMP} shares the same classification of the problem according to (i)--(iv), a rather different phenomenology shows up in the localized setting.
	 \begin{itemize}
	  \item Cases (i) and (ii) are almost identical for the localized and the extended issues if one replaces $\mu_\K$ with $\mu_\G$, up to the fact that in \cite{AST-CMP} there exist graphs supporting ground states at $\mu=\mu_\G$ . Here, on the contrary, the localization of the nonlinearity prevents the existence of ground states.
	  \item In case (iii) the difference with respect to the extended nonlinearity is more remarkable. Indeed, in \cite{AST-CMP}, independently of further properties of the graphs, the critical mass satisfies $\mu_\G=\mu_{\rr^+}$ and the interval of masses for which one has existence of ground states is $(\mu_{\rr^+},\mu_\rr]$. On the contrary, in Theorem \ref{THM1} the reduced critical mass $\mu_\K$ is strictly greater than the critical mass of the half-line and hence the existence interval is smaller. Nevertheless, it is provided to be nonempty, since $\mu_\K<\sqrt{3}<\mu_\rr$. In addition, it is worth mentioning that the fact that $\mu_\K>\mu_{\rr^+}$ allows us to treat also the endpoint case $\mu=\mu_\K$, which is instead open for the extended problem.
	  \item Case (iv) is analogous to the extended case, again with $\mu_\K$ in place of $\mu_\G$. In particular, in both the situations one has to prescribe the (reduced) critical mass to be different from $\mu_\rr$ in order to guarantee the existence of ground states. However, one can see that if $\mu_\K\neq\mu_\rr$, then $\mu_\G<\mu_\K$. Indeed, assume by contradiction $\mu_\G=\mu_\K$, so that $\mathcal{E}_\G(\mu_\K,\K)=\mathcal{E}_\G(\mu_\G)=0$. Item (iv) of Theorem \ref{THM1} entails that there exists a minimizer $u\in H_{\mu_\G}^1(\G)$, and thus $E(u,\K)=E(u)=0$, whence $\|u\|_{6,\K}=\|u\|_{6,\G}$. Hence, $u$ is a ground state of the extended problem, as well, and it is supported on the sole compact core $\K$. Since this is impossible due to standard regularity properties of global minimizers (see \cite{AST-CVPDE}), there results $\mu_\G<\mu_\K$. Consequently, as for case (iii), the existence interval is strictly smaller for the problem with a concentrated nonlinearity.
	 \end{itemize}

	\medskip
	 It is also worth highlighting that, while the fact that the topology of the graph plays a crucial role is a common feature in the context of extended nonlinearities, it is new in the localized setting. Indeed, in the subcritical case the problem of the ground states with localized nonlinearity is affected only by metric properties of the graph (see, e,g, \cite[Theorems 3.3 and 3.4, and Remark 3.8]{T-JMAA}).
	 
	 Nevertheless, it turns out that the metric features preserve their importance even in the critical problem, at least in cases (iii) and (iv), as shown by the following

	\begin{thm}
		\label{THM2}
		Estimates \eqref{EQ-bound mu tilde 1 half-line} are sharp in general; i.e., for every $\varepsilon>0$ there exist two non-compact metric graphs $\G_\varepsilon^1,\G_\varepsilon^2$ (with compact cores $\K^1_\ep,\K^2_\ep$, respectively), with exactly one half-line and no terminal edges, such that 
		\[
		\mu_{\K_\varepsilon^1}\leq \mu_{\rr^+}+\varepsilon\qquad\text{and}\qquad
		\mu_{\K_\varepsilon^2}\geq\sqrt{3}-\varepsilon\,.
		\]
	\end{thm}
	
	\begin{thm}
		\label{THM3}
		Estimates \eqref{EQ-bound mu tilde resto del mondo} are sharp in general; i.e., for every $\varepsilon>0$ there exist two non-compact metric graphs $\G_\varepsilon^1,\G_\varepsilon^2$ (with compact cores $\K^1_\ep,\K^2_\ep$, respectively), without terminal edges and cycle coverings and with at least two half-lines, such that
		\[
		\mu_{\K_\varepsilon^1}\leq\mu_{\rr^+}+\varepsilon\qquad\text{and}\qquad\mu_{\K_\varepsilon^2}\geq\mu_\rr-\varepsilon\,.
		\]
	\end{thm}
	
	Some comments are in order. First, we underline that Theorem \ref{THM2} shows a completely new phenomenon with respect to the extended case. Indeed, in \cite{AST-CMP} for the topologies that fall into case (iii) the actual value of the critical mass $\mu_\G$ is completely insensitive to metric properties of $\G$, as $\mu_\G=\mu_{\rr^+}$, whereas this is not the case when a localization of the nonlinearity occurs. 
	
	Furthermore, we can actually characterize the structure of graphs approaching either $\mu_{\rr^+}$ or $\sqrt{3}$ in Theorem \ref{THM2}. The former case is realized when in the compact core there is a sufficiently long {\em cut edge}, that is a single edge connecting two disjoint subgraph (Proposition \ref{PROP-METRIC TADPOLE}). On the contrary, the latter situation occurs when the compact core is extremely intricate, in the sense that it has a very large total length with respect to its diameter (Proposition \ref{PROP-METRIC INTRICATED K}).
	
	It is also worth mentioning the role of the constant $\sqrt{3}$, which is new with respect to the extended case. For graphs with one half-line and no terminal edges, indeed, it is a kind of critical mass obtained by \eqref{eq-red_mass} if one replaces $C_\K$ with the maximum of $Q$ restricted to functions which are constant on the compact core and exponentially fast decaying on the half-line (see Proposition \ref{PROP- C_6,K geq 1 G with 1 half-line}). Albeit not being optimizers of the problem, these functions seem to play a significant role when the compact core thickens, as non-constants function would necessarily originate a large kinetic energy $\|u'\|_{2,\G}$.
	
	On the other hand, Theorem \ref{THM3} shows a similar asymptotic behaviour induced by the metric. Since this kind of graphs have at least one cut edge in their compact core, by extending or shrinking such an edge, one can recover the two limiting optimal constants (see Propositions \ref{PROP-METRIC SIGNPOST 1}-\ref{PROP-METRIC SIGNPOST 2}). In this case it is an open problem whether an analogous phenomenon could occur for an extended nonlinearity.

	 In conclusion, we recall that, as in the extended case, the existence of negative energy ground states, namely negative energy stationary solutions of \eqref{eq-timeNLS}, is something in sharp contrast with the usual behavior of the time-dependent NLSE on standard domains, where solutions with negative energy blow up in a finite time (see \cite{C}).

	The paper is organized as follows. In Section \ref{sec-prel} we recall some known facts on Gagliardo-Nirenberg inequalities and show some useful results on boundedness from below of $E(\cdot,\K)$ and pre-compactness of the minimizing sequences. In Section \ref{sec-thm1} we prove Theorem \ref{THM1}, stressing the connection between existence results, the value of $C_\K$ and the topology of the graph. Finally, Section \ref{sec-thm23} provides the proof of Theorem \ref{THM2} (Section \ref{sec-iii}) and Theorem \ref{THM3} (Section \ref{sec-iv}), with a particular focus on the relation between $C_\K$ and the metric structure of the graph.


    \section{Preliminary results: Gagliardo-Nirenberg inequalities and ground states}
    \label{sec-prel}
    
    In this section we analyse the connection between the Gagliardo-Nirenberg inequalities and the existence of ground states of \eqref{EQ-minimum problem def}.
    
    First, in addition to \eqref{eq-red_GN} and \eqref{EQ-GN 6}, we recall another well known Gagliardo-Nirenberg inequality on graphs:
    \begin{equation}
    \label{EQ-GN infty}
    \|u\|_{\infty,\G}\leq C_\infty\uLtwo^{1/2}\udot^{1/2}\,,\qquad\forall u\in H^1(\G)\,
    \end{equation}
    (see \cite{AST-JFA} for a proof), where $C_\infty$ denotes the smallest constant for which the inequality is satisfied. It is well known (see e.g. \cite{T-JMAA}) that, if $\G=\rr$, then $C_\infty=1$, while if $\G=\rr^+$, then $C_\infty=\sqrt{2}$, and that in both cases it is attained by $u(x)=e^{-|x|}$. On the other hand, for general non-compact graphs, although it is always true that $C_\infty\leq\sqrt{2}$,  the actual value of this constant depends both on topological and on metric properties of $\G$ . However, if $\G$ is a graph with exactly one half-line, we recover $C_\infty=\sqrt{2}$. Indeed, denoting by $\h$ the unique half-line of $\G$ and setting, for every $\varepsilon>0$,
        \[
        u_\varepsilon(x):=\begin{cases}
        \sqrt{\varepsilon}u(\varepsilon x) & \text{if }x\in\h\\
        \sqrt{\varepsilon} & \text{if }x\in\K
        \end{cases}
        \]
        with $u(x)=e^{-x}$, we get that $u_\varepsilon\in H^1(\G)$ and that
        \[
        \begin{split}
        \sqrt{2}\geq C_\infty\geq\frac{\|u_\varepsilon\|_{\infty,\G}}{\|u_\varepsilon\|_{2,\G}^{1/2}\|u_\varepsilon'\|_{2,\G}^{1/2}}=&\frac{\sqrt{\varepsilon}}{\Big(\|u_\varepsilon\|_{2,\h}^2+\varepsilon|\K|\Big)^{1/4}\|u_\varepsilon'\|_{2,\h}^{1/2}}\\
        =&\frac{\sqrt{\varepsilon}}{\Big(\frac{1}{2}+\varepsilon|\K|\Big)^{1/4}\left(\tfrac{\ep^2}{2}\right)^{1/4}
        }\longrightarrow\sqrt{2},\qquad\text{as}\quad\varepsilon\to0\,.
    	\end{split}
        \]
        
 Another important preliminary remark concerns a modified version of the Gagliardo-Nirenberg inequality \eqref{EQ-GN 6}, which has been first established in \cite[Lemma 4.4]{AST-CMP}. Let $\G$ satisfy \textbf{(A)} and assume also that it does not contain any terminal edge. Then, if $\mu\in(0,\mu_\rr]$, for every $u\in\HmuG$ there exists $\theta_u\in[0,\mu]$ (depending on $u$) such that
    
    \begin{equation}
    \label{EQ-modified GN}
    \uLsix^6\leq3\Big(\frac{\mu-\theta_u}{\mu_\rr}\Big)^2\udot^2+C\sqrt{\theta_u},
    \end{equation}
with $C>0$ depending only on $\G$. In addition, note that if $\{u_n\}\subset\HmuG$ is a sequence of functions such that $E_\K(u_n,\G)\leq-\alpha<0$, for some $\alpha>0$, then $\inf_{n}\theta_{u_n}>0$. Indeed, by \eqref{EQ-modified GN}
    	
    	\begin{equation}
    	\label{EQ-compactness proof 1}
    	\frac{1}{2}\|u_n'\|_{2,\G}^2\Big(1-\frac{(\mu-\theta_{u_n})^2}{\mu_\rr^2}\Big)-\frac{C}{6}\sqrt{\theta_{u_n}}\leq E(u_n,\K)\leq-\alpha<0,
    	\end{equation}
    	so that $\theta_{u_n}>0$ uniformly on $n$.  

    \medskip
    As mentioned in the Introduction, the key Gagliardo-Nirenberg inequality of the problem with localized nonlinearity is not the standard one given by \eqref{EQ-GN 6}, but instead \eqref{eq-red_GN} where the $L^6$ term affects only the compact core of the graph. In particular, it is evident by \eqref{eq-red_mass} that the value of the best constant $C_{\K}$ is the crucial parameter in order to determine whether solutions of \eqref{EQ-minimum problem def} exist or not. Indeed, plugging \eqref{eq-red_GN} into \eqref{EQ-energy def},
    \begin{equation}
    	\label{EQ-bound energy with GN}
    	E(u,\K)\geq\frac{1}{2}\udot^2\left(1-\frac{C_{\K}}{3}\mu^2\right)\,\qquad\forall u\in\HmuG
    \end{equation}
    showing that $C_\K$ plays a role in establishing the lower boundedness of the energy. More precisely, we can state the following
    
    \begin{lem}
    	\label{LEM-infimum and critical mass}
    	Let $\G$ satisfy \textbf{(A)} and $\mu_\K$ be the reduced critical mass defined by \eqref{eq-red_mass}-\eqref{EQ-def C_6.K}. The following classification holds:
    	\begin{itemize}
    		\item[(i)] if $\mu\leq\mu_\K$, then $\mathcal{E}_\G(\mu,\K)=0$;\\[-.2cm]
    		\item[(ii)] if $\mu\in(\mu_\K,\mu_\rr]$, then $\mathcal{E}_\G(\mu,\K)<0$; \\[-.2cm]
    		\item[(iii)] if $\mu>\mu_\rr$, then $\EEK=-\infty$.
    	\end{itemize}
    \end{lem}
	
	\begin{proof}
		First, note that, for every $\mu>0$, if $\{u_n\}\subset\HmuG$ is a sequence such that $\|u_n'\|_{2,\G}\to0$, then
		\[
		\EEK\leq E(u_n,\K)\leq\frac{1}{2}\|u_n'\|_{2,\G}^2\longrightarrow0\qquad\text{as}\quad n\to\infty\,
		\]
		so that $\EEK\leq0$. Furthermore, if $\mu\leq\mu_\K$, then \eqref{EQ-bound energy with GN} entails that $E(u,\K)\geq0$ for all $u\in\HmuG$, thus proving (i).
		
		On the other hand, assume $\mu>\mu_\K$; for instance, $\mu=(1+\delta)\mu_\K$, for some $\delta>0$. Now, by \eqref{EQ-def C_6.K}, there exists $u\in\HmuG$ such that
		\[
		Q(u)>\frac{C_\K}{(1+\delta)^2},
		\]
		whence
		\[
		E(u,\K)<\frac{1}{2}\udot^2\Big(1-\frac{C_\K\mu^2}{3(1+\delta)^2}\Big)=0
		\]
        by \eqref{eq-red_mass} and the assumption on $\mu$, which yields (ii).
		
		Finally, let $\mu>\mu_\rr$ and $v\in H_\mu^1(\rr)$ such that $\text{supp}\{v\}=[0,1]$ and $E(v,\rr)<0$ (the existence of $v$ being guaranteed by \eqref{EQ-inf rr}). For every $\lambda>0$ define then
		\[
		v_\lambda(x):=\sqrt{\lambda}v(\lambda x)\,,
		\]
so that, $v_\lambda\in H_\mu^1(\rr)$ and $\text{supp}\{v_\lambda\}=[0,1/\lambda]$. Clearly, when $\lambda$ is large enough, $v_\lambda$ can be regarded as a function on $\G$ supported on a single edge of the compact core $\K$. As a consequence,
		\[
		\EEK\leq E(v_\lambda,\K)=\lambda^2 E(v,\rr)\longrightarrow-\infty\qquad\text{as}\quad\lambda\to\infty,
		\] 
		which concludes the proof.
	\end{proof}

	As boundedness from below is not enough in order to prove ground states to exist, the following lemma provides some sufficient conditions for existence/nonexistence.
	
	\begin{lem}
		\label{LEM-ground states}
		Let $\G$ satisfy \textbf{(A)} and $\mu_\K$ be the reduced critical mass defined by \eqref{eq-red_mass}--\eqref{EQ-def C_6.K}. Then: 
		\begin{itemize}
			\item[(i)] whenever $\mu<\mu_\K$, $\EEK$ is not attained;\\[-.2cm]
			\item[(ii)] whenever $\mu\in(\mu_\K,\mu_\rr]$ and $\EEK>-\infty$, $\EEK$ is attained.
		\end{itemize}
	\end{lem}

	\begin{proof}
		If $\mu<\mu_\K$, then \eqref{eq-red_mass}--\eqref{EQ-bound energy with GN} imply that $E(u,\K)>0$ for every $u\in\HmuG$ and, combining with Lemma \ref{LEM-infimum and critical mass}, (i) follows immediately.
		
		On the other hand, suppose $\mu\in(\mu_\K,\mu_\rr]$ and $\EEK>-\infty$. Let, also, $\{u_n\}\subset\HmuG$ be a minimizing sequence. By Lemma \ref{LEM-infimum and critical mass}, there exists $\alpha>0$ such that $E(u_n,\K)\leq-\alpha$ as $n$ is large enough, so that by \eqref{EQ-compactness proof 1}, \eqref{EQ-modified GN} holds with $\theta_{u_n}\geq C>0$. Consequently, $\mu-\theta_{u_n}<\mu_\rr$, so that \eqref{EQ-modified GN} entails that $\{u_n\}_{n\in\mathbb{N}}$ is bounded in $H^1(\G)$ and that $u_n\rightharpoonup u$ in $H^1(\G)$ and $u_n\to u$ in $L^6(\K)$, for some $u\in H^1(\G)$. Moreover, by weak lower semicontinuity
		\begin{equation}
		\label{EQ-compatness proof 2}
		E(u,\K)\leq\liminf_{n}E(u_n,\K)=\EEK\,.
		\end{equation}
		
		Therefore, we are left to prove that $\|u\|_{2,\G}^2=:m=\mu$. By weak lower semicontinuity again, $m\leq\mu$. On the other hand, it is immediate to see that $m\neq0$, since otherwise $u\equiv0$ and, by \eqref{EQ-compatness proof 2}, $\EEK\geq0$, which is a contradiction. Moreover, if $m<\mu$, then there exists $\sigma>1$ such that $\|\sigma u\|_{2,\G}^2=\mu$ and 
		\[
		E(\sigma u,\K)=\frac{\sigma^2}{2}\udot^2-\frac{\sigma^6}{6}\uLsixcompact^6<\sigma^2 E(u,\K)<E(u,\K),
		\]
		which contradicts \eqref{EQ-compatness proof 2} and hence concludes the proof.
    \end{proof}

   \begin{rem}
   	Lemma \ref{LEM-ground states} does not say anything about $\mu=\mu_\K$. Indeeds, this endpoint problem strongly depends on the topology of the graph and must be addressed case by case.
   \end{rem}
   

	\section{Proof of Theorem \ref{THM1}: how the topology affects $C_{\K}$}
	\label{sec-thm1}
	
	Throughout this section we provide the proof of Theorem \ref{THM1}. It is based on the study of the actual value of $C_\K$ in the four distinct topological classes (i)-(iv) defined by Theorem \ref{THM1}.
	
	As a first step, we derive upper and lower bounds for $C_\K$ on a generic non-compact graph. In the following, we denote by $C_{\rr^+}$ and $C_\rr$ the best constant of \eqref{EQ-GN 6} when $\G=\rr^+$ and $\G=\rr$, respectively. 
	
	\begin{prop}
		\label{PROP-C_6,K between rr and rr+ }
		For every $\G$ satisfying \textbf{(A)}, the constant $C_{\K}$ fulfills
		\begin{equation}
			\label{EQ-C_6,K between rr and rr+ }
			C_\rr\leq C_{\K}\leq C_{\rr^+}\,.
		\end{equation}
	\end{prop}
	
	\begin{proof}
        Let $u\in H^1(\G)$ and assume with no loss of generality $u\geq0$. Then, denote by $u^*$ its decreasing rearrangement on $\rr^+$, namely
        \begin{equation}
         \label{eq-decrear}
         u^*(x):=\inf\{t\geq0:\rho(t)\leq x\}\quad x\in[0,|\G|),\qquad\text{with}\qquad \rho(t):=\sum_{e\in\mathrm{E}}|\{x_e\in I_e:u_e(x_e)>t\}|\quad t\geq0.
        \end{equation}
        By well known properties of rearrangements (see \cite[Section 3]{AST-CVPDE} for details),
		\begin{equation}
			\label{EQ-properties rearrangements}
			\|(u^*)'\|_{2,\rr^+}\leq\udot\,,\qquad\|u^*\|_{p,\rr^+}=\uLp\qquad\forall\,p\geq1\,,
		\end{equation}
		so that
		\[
		Q(u)\leq\frac{\uLsix^6}{\uLtwo^4\udot^2}\leq\frac{\|u^*\|_{6,\rr^+}^6}{\|u^*\|_{2,\rr^+}^4\|(u^*)'\|_{2,\rr^+}^2}\leq C_{\rr^+},
		\]
		and recalling \eqref{EQ-def C_6.K}
		\[
		C_\K\leq C_{\rr^+}\,.
		\]
		On the other hand, there exists a sequence $\{\wt{v}_n\}\subset H^1(\rr)$, $\wt{v}_n\geq0$, such that
		\begin{equation}
		 \label{eq-GNr}
		\frac{\|\wt{v}_n\|_{6,\rr}^6}{\|\wt{v}_n\|_{2,\rr}^4\|\wt{v}_n'\|_{2,\rr}^2}\longrightarrow C_\rr\qquad\text{as}\quad n\to\infty\,
		\end{equation}
		and such that $\|\wt{v}_n\|_{2,\rr}^4\|\wt{v}_n'\|_{2,\rr}^2\geq C>0$ (see, e.g. \cite{DELL-JLMS}). Now, simply truncating and lowering $\wt{v}_n$, one can define another sequence $v_n\in H^1(\rr)$, which is compactly supported, such that $\|v_n\|_{2,\rr}\leq\|\wt{v}_n\|_{2,\rr}$, $\|v_n'\|_{2,\rr}\leq\|\wt{v}_n'\|_{2,\rr}$ and $\|v_n-\wt{v}_n\|\to0$. As a consequence,
		\[
		 \frac{\|\wt{v}_n\|_{6,\rr}^6}{\|\wt{v}_n\|_{2,\rr}^4\|\wt{v}_n'\|_{2,\rr}^2}\leq\frac{\|v_n\|_{6,\rr}^6}{\|v_n\|_{2,\rr}^4\|v_n'\|_{2,\rr}^2}+\frac{\|v_n-\wt{v}_n\|_{6,\rr}^6}{\|\wt{v}_n\|_{2,\rr}^4\|\wt{v}_n'\|_{2,\rr}^2}\leq C_\rr+o(1),
		\]
		and thus \eqref{eq-GNr} still holds with $\wt{v}_n$ replaced by $v_n$. Then, for a fixed edge $e$ in $\K$, we can define $u_n\in H^1(\G)$ as
		\[
		u_n(x):=\begin{cases}
		\sqrt{\lambda_n}v_n(\lambda_n x) & \text{if }x\in I_e\\
		0 & \text{elsewhere on }\G\,,
		\end{cases}
		\]
		where $\lambda_n>0$ is chosen so that $\text{supp}\{u_n\}\subset I_e$, and since 
		\[
		C_\K\geq Q(u_n)=\frac{\|v_n\|_{6,\rr}^6}{\|v_n\|_{2,\rr}^4\|v_n'\|_{2,\rr}^2}\,,
		\]
		passing to the limit one obtains $C_\K\geq C_\rr$.
	\end{proof}

	In general, inequalities in \eqref{EQ-C_6,K between rr and rr+ } are not sharp. However, in order to prove this, it is necessary to preliminarily detect under which assumptions the optimal constant is attained.

	\begin{lem}
		\label{LEM-C K attained}
		Let $\G$ satisfy \textbf{(A)} and assume also that it does not contain any terminal edge. If, in addition, $C_\K\neq C_\rr$, then there exists $0\not\equiv u\in H^1(\G)$ such that $Q(u)=C_\K$.
	\end{lem}

	\begin{proof}
		Let $\{u_n\}\subset H^1(\G)$ be a maximizing sequence for $C_{\K}$, so that there exists $\varepsilon_n\downarrow0$ such that
		\[
		\|u_n\|_{6,\K}^6=(C_{\K}-\varepsilon_n)\|u_n\|_{2,\G}^4\|u_n'\|_{2,\G}^2\,.
		\]
		Moreover, as the functional $Q$ is homogeneous, we can set without loss of generality
		\[
		\|u_n\|_{2,\G}^2=:\mu_n=\sqrt{\frac{3}{C_{\K}-\varepsilon_n}}\,,
		\]
		so that
		\begin{equation}
			\label{EQ-proof C K attained quotient = 3}
			\frac{\|u_n\|_{6,\K}^6}{\|u_n'\|_{2,\G}^2}=3\,.
		\end{equation}
		Note also that, as $\varepsilon_n\to0$ and $C_{\K}>C_\rr$ by assumption, (for large $n$) there results
		\[
		\mu_n<\mu_\rr+\delta\,,
		\]
		for some fixed $\delta>0$. Therefore, combining \eqref{EQ-proof C K attained quotient = 3} and \eqref{EQ-modified GN},
		\begin{equation*}
		\|u_n'\|_{2,\G}^2=\frac{1}{3}\|u_n\|_{6,\K}^6\leq\frac{1}{3}\|u_n\|_{6,\G}^6\leq\frac{\mu_n^2}{\mu_\rr^2}\|u_n'\|_{2,\G}^2+C\sqrt{\mu_\rr}
		\end{equation*}
		and rearranging terms
		\begin{equation*}
		\Big(1-\frac{\mu_n^2}{\mu_\rr^2}\Big)\|u_n'\|_{2,\G}^2\leq C\sqrt{\mu_\rr}\,.
		\end{equation*}
		Hence, $\{u_n\}$ is bounded in $H^1(\G)$ and there exists $u\in H^1(\G)$ such that
		\begin{equation}
		\label{EQ-u_n to u strong L6}
		 u_n\rightharpoonup u\quad\text{in}\quad H^1(\G)\qquad\text{and}\qquad\|u_n\|_{p,\K}\to\|u\|_{p,\K},\qquad\forall p\in[1,\infty].
		\end{equation}
        In addition, by weak lower semicontinuity
		\begin{equation}
		\label{EQ-semicontinuity u_n}
		\uLtwo\leq\liminf_{n}\|u_n\|_{2,\G}\qquad\text{and}\qquad\udot\leq\liminf_{n}\|u_n'\|_{2,\G}.
		\end{equation}
		On the other hand, by \eqref{EQ-GN infty},
		\begin{equation*}
		\|u_n'\|_{2,\G}^2=\frac{1}{3}\|u_n\|_{6,\K}^6\leq\frac{1}{3}\|u_n\|_{\infty,\G}^6|\K|\leq\frac{1}{3} C_\infty^6\|u_n\|_{2,\G}^3\|u_n'\|_{2,\G}^3|\K|
		\end{equation*}
		and thus 
		\begin{equation}
		\label{EQ-u_n' norm bounded away 0}
		\|u_n'\|_{2,\G}\geq\frac{3}{|\K|C_\infty^6\mu_n^3}\geq\frac{3}{|\K| C_\infty^6(\mu_\rr)^{3/2}}>0\,.
		\end{equation}
		Assume then $u\equiv0$ on $\G$. As $\|u_n\|_{\infty,\K}\to0$, we have
		\begin{equation*}
		\|u_n'\|_{2,\G}^2\leq\frac{1}{3}\|u_n\|_{6,\K}^6\leq\frac{1}{3}\|u_n\|_{\infty,\K}^6|\K|\longrightarrow0\qquad\text{as}\quad n\to\infty,
		\end{equation*}
		which contradicts \eqref{EQ-u_n' norm bounded away 0}. As a consequence, $u\not\equiv0$ and, combining with \eqref{EQ-u_n to u strong L6} and \eqref{EQ-semicontinuity u_n} 
		\begin{equation*}
		C_{\K}\geq Q(u)\geq\limsup_{n}Q(u_n)=C_{\K},
		\end{equation*}
		which concludes the proof.
	\end{proof}
	
	Now, we can improve the estimates in Proposition \ref{PROP-C_6,K between rr and rr+ }, owing to the topological features of the graphs. Precisely, the following proposition distinguishes the cases when the graph does possess a terminal edge, i.e. (i) of Theorem \ref{THM1}, and when it does not, i.e. (ii)-(iv) of Theorem \ref{THM1}. Furthermore, for cases (i) and (ii) it provides the exact value of $C_\K$. 
	
	\begin{prop}
		\label{PROP-C_6,K with respect to rr+}
		For every $\G$ satisfying \textbf{(A)}: 
		\begin{itemize}
			\item[(i)] if $\G$ has (at least) a terminal edge, then $C_{\K}=C_{\rr^+}$;\\[-.2cm]
			\item[(ii)] if $\G$ has no terminal edge, then $C_{\K}<C_{\rr^+}$; if $\G$ admits also a cycle-covering, then $C_{\K}=C_\rr$.
		\end{itemize}
	\end{prop}

	\begin{proof}
        Let us split the proof in two parts.
        
		\textit{Part (i): graphs with a terminal edge.} Arguing as in the proof of Proposition \ref{PROP-C_6,K between rr and rr+ }, one can see that there exists $\{v_n\}\subset H^1(\rr^+)$, $v_n(x)\equiv0$ for large $x$, such that
		\[
		\frac{\|v_n\|_{6,\rr^+}^6}{\|v_n\|_{2,\rr^+}^4\|v_n'\|_{2,\rr^+}^2}\longrightarrow C_{\rr^+},\qquad\text{as}\quad n\to\infty\,,
		\]
		and that from $\{v_n\}$ one can construct a sequence $\{u_n\}\subset H^1(\G)$, supported only on a terminal edge of $\G$, such that
		\[
		Q(u_n)=\frac{\|v_n\|_{6,\rr^+}^6}{\|v_n\|_{2,\rr^+}^4\|v_n'\|_{2,\rr^+}^2}.
		\]
		Hence, as $C_\K\geq Q(u_n)$, passing to the limit in $n$ yields $C_\K\geq C_{\rr^+}$, which in view of \eqref{EQ-C_6,K between rr and rr+ } completes the proof.
		
		\textit{Part (ii): graphs without terminal edges.} Owing to Proposition \ref{PROP-C_6,K between rr and rr+ }, it is sufficient to show that $C_{\K}\neq C_{\rr^+}$. Suppose by contradiction that $C_{\K}=C_{\rr^+}$. By Lemma \ref{LEM-C K attained} $C_{\K}$ is attained, i.e. there exists $u\in H^1(\G)$ such that $Q(u)=C_{\rr^+}$. It follows by \eqref{EQ-properties rearrangements} that
		\begin{equation}
			\label{EQ-C rr^+ not attained calculations}
			C_{\rr^+}=Q(u)\leq\frac{\uLsix^6}{\uLtwo^4\udot^2}\leq\frac{\|u^*\|_{6,\rr^+}^6}{\|u^*\|_{2,\rr^+}^4\|(u^*)'\|_{2,\rr^+}^2}\leq C_{\rr^+},
		\end{equation}
		with $u^*$ being the decreasing rearrangement of $u$ (which is again not restrictive to assume nonnegative). Thus
		\[
			\uLsixcompact=\uLsix\,,
		\]
		so that $u\equiv0$ outside of $\K$. As a consequence, $(u_{\mid_\K})^*\in H^1(\rr^+)$ is null for large $x$ and at the same time attains $C_{\rr^+}$, which is impossible. Then $C_\K\neq C_{\rr^+}$.  
		
		Suppose, now, that $\G$ admits a cycle-covering. Given $u\in H^1(\G)$, $u\geq0$, denote by $\widehat{u}\in H^1(\rr)$ its symmetric rearrangement, i.e.
		\[
		 \widehat{u}(x):=\inf\{t\geq0:\rho(t)\leq 2|x|\}\qquad x\in(-|\G|/2,|\G|/2),
		\]
		with $\rho$ defined by \eqref{eq-decrear}. The presence of a cycle-covering entails that $u$ attains at least twice almost every value in its image, and thus, from \cite[Section 3]{AST-CVPDE},
		\[
			\|(\widehat{u})'\|_{2,\rr}\leq\udot\,,\qquad \|\widehat{u}\|_{p,\rr}=\uLp\qquad\forall\,p\geq1\,.
		\]
		Hence
		\[
		Q(u)\leq\frac{\|\widehat{u}\|_{6,\rr}^6}{\|\widehat{u}\|_{2,\rr}^4\|(\widehat{u})'\|_{2,\rr}^2}\leq C_\rr\,,
		\]
		and, passing to the supremum on $H^1(\G)$, in view of \eqref{EQ-C_6,K between rr and rr+ } we have $C_\K=C_\rr$.
	\end{proof}

	The next proposition pushes forward the analysis of $C_\K$ in the case of graphs with exactly one half-line and no terminal edge.
	
	\begin{prop}
		\label{PROP- C_6,K geq 1 G with 1 half-line}
		For every $\G$ satisfying \textbf{(A)} with exactly one half-line and no terminal edge
		\[
			C_{\K}>1\,.
		\]
	\end{prop}
	
	\begin{proof}
	    We divide the proof in two steps.
	    
	    \emph{Step(i): $C_\K\geq1$}. Denote by $\h$ the half-line of $\G$ and $L$ the measure of $\K$. For every $c,\alpha>0$, we define the function $u_{c,\alpha}\in H^1(\G)$ such that
		\begin{equation}
		\label{EQ-def u c alpha}
			u_{c,\alpha}(x):=\begin{cases}
			c & \text{if }x\in\K\\
			ce^{-\alpha x} & \text{if }x\in\h\,,
			\end{cases}
		\end{equation}
		(see for instance Figure \ref{FIGURE-constant and exponential}).
		
		\begin{figure}[t]
			\centering
			\includegraphics[width=0.6\textwidth]{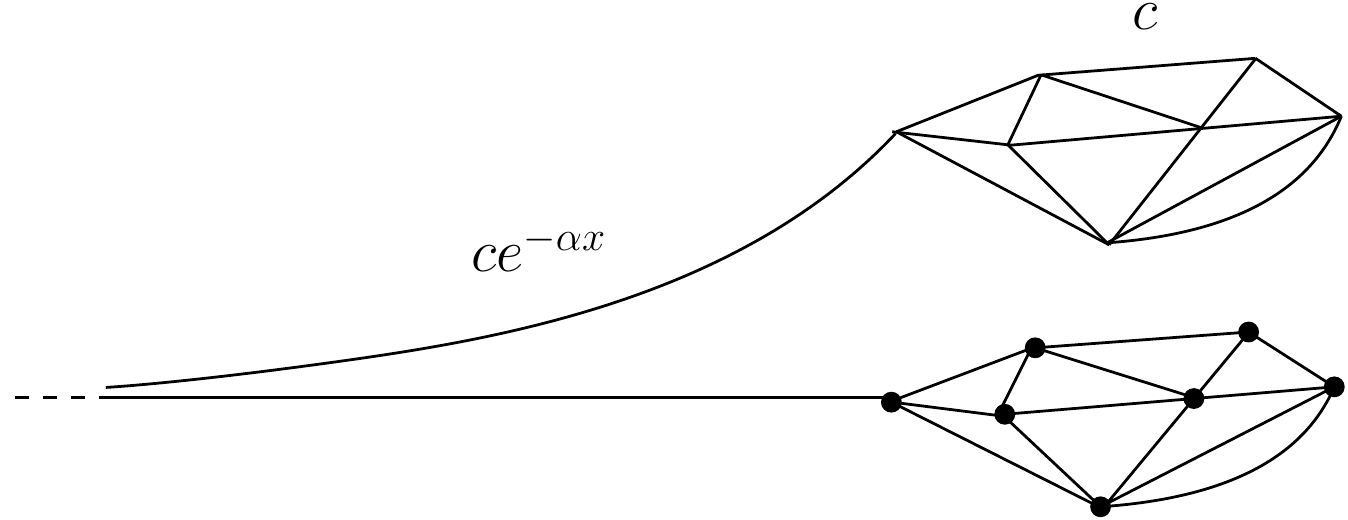}
			\caption{example of the function $u_{c,\alpha}$ as in the proof of Proposition \ref{PROP- C_6,K geq 1 G with 1 half-line}.}
			\label{FIGURE-constant and exponential}
		\end{figure}
		
		Direct computations yield
		\[
			\begin{split}
			\|u_{c,\alpha}\|_{2,\G}^2=&\int_\K c^2\,dx+\int_\h c^2e^{-2\alpha x}\,dx=c^2 L+\frac{c^2}{2\alpha}=c^2\Big(L+\frac{1}{2\alpha}\Big)\\
			\|u_{c,\alpha}'\|_{2,\G}^2=&\int_\h c^2\alpha^2e^{-2\alpha x}\,dx=\frac{c^2\alpha}{2}\\
			\|u_{c,\alpha}\|_{6,\K}^6=&\int_\K c^6\,dx=c^6 L\,,
			\end{split}
		\]
        so that
		\[
			Q(u_{c,\alpha})=\frac{c^6 L}{c^4\Big(L+\frac{1}{2\alpha}\Big)^2\frac{c^2\alpha}{2}}=\frac{8\alpha L}{(2\alpha L+1)^2}=:F(\alpha)\,.
		\]
		Now,
		\begin{equation*}
			F'(\alpha)=\frac{8L (1-2L\alpha)}{(2\alpha L+1)^3}
		\end{equation*}
		and then $F$ has a maximum point at $\alpha=\overline{\alpha}:=\frac{1}{2L}$. Hence, for every $c,\alpha>0$
		\begin{equation}
		\label{EQ-quotient u c alpha leq 1}
			Q(u_{c,\alpha})\leq F(\overline{\alpha})=1,
		\end{equation}
		and thus, by definition,
		\[
		C_\K\geq1\,.
		\]
		
		\emph{Step(ii): $C_\K\neq1$}. Assume, by contradiction, that $C_\K=1$. Since $1=F(\overline{\alpha})=Q(u_{c,\overline{\alpha}})$, $u_{c,\overline{\alpha}}$ as in \eqref{EQ-def u c alpha} with $\alpha=\overline{\alpha}:=\frac{1}{2L}$ is an optimizer of the functional $Q$ on $H^1(\G)\backslash\{0\}$. As a consequence, for every $\varphi\in H^1(\G)$,
		\[
		 \f{d}{d\ep}Q(u_{c,\overline{\alpha}}+\ep\varphi)_{\big|_{\ep=0}}=0,
		\]
		namely, after standard computations
		\[
		 A(u_{c,\overline{\alpha}})\int_\G u_{c,\overline{\alpha}}'\,\varphi'\dx+B(u_{c,\overline{\alpha}})\int_\G u_{c,\overline{\alpha}}\,\varphi\dx=C(u_{c,\overline{\alpha}})\int_\K|u_{c,\overline{\alpha}}|^4\,u_{c,\overline{\alpha}}\,\varphi\dx,
		\]
		with
		\begin{gather*}
		 A(u_{c,\overline{\alpha}}):=\f{2\|u_{c,\overline{\alpha}}\|_{6,\K}^6}{\big(\|u_{c,\overline{\alpha}}\|_{2,\G}\|u_{c,\overline{\alpha}}'\|_{2,\G}\big)^4},\qquad
		 B(u_{c,\overline{\alpha}}):=\f{4\|u_{c,\overline{\alpha}}\|_{6,\K}^6}{\|u_{c,\overline{\alpha}}\|_{2,\G}^6\|u_{c,\overline{\alpha}}'\|_{2,\G}^2}\\[.2cm]
		 C(u_{c,\overline{\alpha}}):=\f{6}{\|u_{c,\overline{\alpha}}\|_{2,G}^4\|u_{c,\overline{\alpha}}'\|_{2,\G}^2}.
		\end{gather*}
		Now, arguing as in the proof of \cite[Proposition 3.3]{AST-CVPDE}, we get that
		\begin{equation}
		 \label{eq-ELGN}
		 -A(u_{c,\overline{\alpha}})\,(u_{c,\overline{\alpha}})_e''+B(u_{c,\overline{\alpha}})\,(u_{c,\overline{\alpha}})_e=\chi_{\K}\,C(u_{c,\overline{\alpha}})\,|(u_{c,\overline{\alpha}})_e|^4\,(u_{c,\overline{\alpha}})_e,\qquad\forall e\in\mathrm{E},
		\end{equation}
		and that $u_{c,\overline{\alpha}}\in\mathrm{dom}(-\Delta_\G)$. However, this is impossible, since the Kirchhoff conditions
		\[
		 \sum_{e\succ v}\f{d(u_{c,\overline{\alpha}})_e}{dx_e}(\vv)=0,\qquad\forall\vv\in\K,
		\]
		are not fulfilled by $u_{c,\overline{\alpha}}$ at the vertex joining the compact core to the half-line. As a consequence, $u_{c,\overline{\alpha}}$ cannot be an optimizer of $Q$, so that $C_\K\neq1$.
	\end{proof}

	\begin{rem}
		\label{rem-cost+roba}
		Note that, whenever $\G$ is a graph with exactly one half-line and no terminal edges,
		\[
		Q(u)\leq 1
		\]
		for every $u\in H^1(\G)$ which is constant on the compact core, independently of its specific form on the half-line $\h$. Indeed, letting $c:=u_{|_\K}$ and $m:=\|u\|_{2,\h}^2$, Proposition 4.3 in \cite{T-JMAA} ensures that 
		\[
		\inf\,\{\,\|v'\|_{2,\h}\,:\,v\in H^1(\h),\,v(0)=c\text{ and }\|v\|_{2,\h}^2=m\,\}
		\]
		is attained by the exponential function $\varphi(x)=ce^{-c^2\frac{x}{m}}$. Thus, 
		\[
		Q(u)=\frac{c^6|\K|}{(c^2|\K|+m)^2\|u'\|_{2,\h}^2}\leq\frac{c^6|\K|}{(c^2|\K|+m)^2\|\varphi'\|_{2,\h}^2}\leq 1
		\]
by the proof of the previous proposition.
	\end{rem}

	We are now ready to prove Theorem \ref{THM1}.
	
	\begin{proof}[Proof of Theorem \ref{THM1}]
	The first part of Theorem \ref{THM1}, i.e. the validity of \eqref{EQ- THM1 inf}, is proved by Lemma \ref{LEM-infimum and critical mass} setting $\mu_\K$ as in \eqref{eq-red_mass}-\eqref{EQ-def C_6.K}. On the other hand, the values and the estimates on $\mu_\K$ in cases (i)-(iv) are a straightforward application of the results of Propositions \ref{PROP-C_6,K between rr and rr+ }, \ref{PROP-C_6,K with respect to rr+} and \ref{PROP- C_6,K geq 1 G with 1 half-line}.
	
	It is then left to discuss the sole existence of the ground states. Since for $\mu>\mu_\rr$ nonexistence is immediate by \eqref{EQ- THM1 inf} and for $\mu<\mu_\K$ it is straightforward by Lemma \ref{LEM-ground states}, we only consider masses $\mu_\K\leq\mu\leq\mu_\rr$. Let us split the proof according to the four classes.
	
	\emph{Case (i): graphs with at least a terminal edge.} Recall that in this case $\mu_\K=\mu_{\rr^+}$. Assume first that $\mu>\mu_{\rr^+}$. It is well known that there exists $v\in H^1(\rr^+)$, with $v(x)\equiv0$ for large $x$, such that $E(v,\rr^+)<0$. Then, arguing as in the proof of Lemma \ref{LEM-infimum and critical mass}, one can construct a sequence $\{u_n\}\subset\HmuG$ supported on a terminal edge of $\G$ for which $E(u_n,\K)\to-\infty$, as $n\to\infty$. Hence, $\EEK=-\infty$ and no ground state may exist.
	
    Now, assume $\mu=\mu_{\rr^+}$. Suppose also, by contradiction, that there exists a ground state $u\in H_{\mu_{\rr^+}}^1(\G)$. Since, by Lemma \ref{LEM-infimum and critical mass}, $E(u,\K)=\mathcal{E}_\G(\mu_{\rr^+},\K)=0$, we get
	\[
	\frac{\uLsixcompact^6}{\udot^2}=3
	\]
	and, dividing both terms by $\mu_{\rr^+}^2$,
	\[
	Q(u)=\frac{3}{\mu_{\rr^+}^2}=C_{\rr^+}\,.
	\]
	Hence, since $C_\K=C_{\rr^+}$ by Proposition \ref{PROP-C_6,K with respect to rr+}, $u$ is an optimizer for $C_\K$ too. However, arguing as in \eqref{EQ-C rr^+ not attained calculations} and the subsequent paragraph, this entails that the decreasing rearrangement on $\rr^+$ of $u_{\mid_\K}$ is a compactly-supported function attaining $C_{\rr^+}$, which is well known to be impossible. Hence, ground states do not exist also when $\mu=\mu_{\rr^+}$.
	
	\emph{Case (ii): graphs admitting a cycle-covering.} Recall that in this case $\mu_\K=\mu_{\rr}$, which is then the unique value of the mass to discuss. This case can be dealt with repeating the previous argument with $\mu_\rr$ in place of $\mu_{\rr^+}$, $C_\rr$ in place of $C_\rr^+$ and the symmetric rearrangement $\widehat{u}$ in place of the decreasing rearrangement $u^*$.
	
	\emph{Case (iii): graphs with exactly one half-line.} Recall that in this case $\mu_\K\in(\mu_{\rr^+},\sqrt{3})$. By \eqref{EQ-modified GN} we get
	\[
	E(u,\K)\geq\frac{1}{2}\udot^2\Big(1-\frac{\mu^2}{\mu_\rr^2}\Big)-\frac{C}{6}\sqrt{\mu}\,,
	\] 
	and, since $\sqrt{3}<\mu_\rr$, $\EEK>-\infty$. Then, if $\mu>\mu_\K$, by Lemma \ref{LEM-ground states} ground states exist. 
	
	On the other hand, assume $\mu=\mu_\K$ and let $u\in H_{\mu_\K}^1(\G)$ be an optimizer of $C_\K$ provided by Lemma \ref{LEM-C K attained} (note that here $C_\K>1>C_\rr$). Then
	\[
	\frac{\uLsixcompact^6}{\udot^2}=C_\K\mu_\K^2=3,
	\]
	so that, recalling \eqref{EQ- THM1 inf},
	\[
	E(u,\K)=\frac{1}{2}\udot^2\Big(1-\frac{\uLsixcompact^6}{3\udot^2}\Big)=0=\mathcal{E}_\G(\mu_\K,\K),
	\]
	which entails that $u$ is a ground state.
	
	\emph{Case (iv): graphs without terminal edges and cycle-coverings, and with at least two half-lines.} Recall that in this case $\mu_\K\in(\mu_{\rr^+},\mu_\rr]$. Here, if one assumes $\mu_\K\neq\mu_\rr$ (that is, $C_\K\neq C_\rr$), then it is possible to recover the argument in the proof of case (iii).
	\end{proof} 
	

	\section{Proofs of Theorems \ref{THM2}-\ref{THM3}: how the metric affects $C_{\K}$}
	\label{sec-thm23}
	
	In this section, we present the proofs of Theorems \ref{THM2} (Section \ref{sec-iii}) and \ref{THM3} (Section \ref{sec-iv}). Precisely, we show that in cases (iii) and (iv) of Theorem \ref{THM1}, the estimates on the reduced critical mass $\mu_\K$ (\eqref{EQ-bound mu tilde 1 half-line} and \eqref{EQ-bound mu tilde resto del mondo}, respectively) cannot be improved in general. In particular, we explain how upper/lower bounds given by \eqref{EQ-bound mu tilde 1 half-line} and \eqref{EQ-bound mu tilde resto del mondo} can be asymptotically obtained suitably modifying the metric of the compact core of the graph.
	
	As in the previous section, the bulk of the analysis focuses on $C_\K$. In addition, we will tacitly assume, in the following, that any graph satisfies \textbf{(A)}.
	
	
	\subsection{Graphs with exactly one half-line and no terminal edges}
	\label{sec-iii}
	
	We first focus on graphs with exactly one half-line and no terminal edges, for which we already proved that $\mu_{\rr^+}<\mu_\K<\sqrt{3}$, i.e. $1<C_\K<C_{\rr^+}$, and that ground states do exist if and only if $\mu\in[\mu_\K,\mu_\rr]$. 
	
	The first preliminary result concerns the existence of a sequence of graphs whose optimal constants converge to $C_{\rr^+}$.
	
	\begin{figure}[t]
		\centering
		\includegraphics[width=0.6\textwidth]{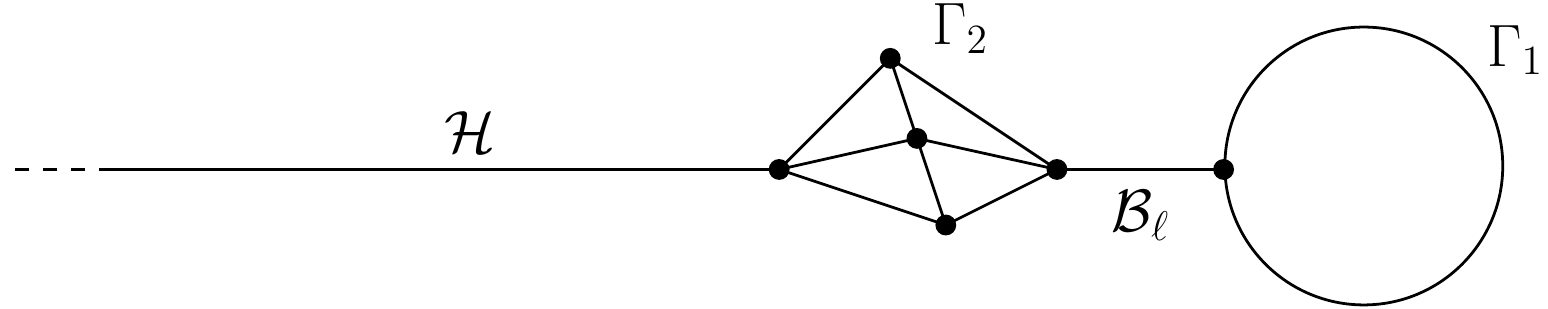}
		\caption{The graph $\G_\ell$ of Proposition \ref{PROP-METRIC TADPOLE}}
		\label{FIGURE-one halfline metric.}
	\end{figure}
	
	\begin{prop}
		\label{PROP-METRIC TADPOLE}
		 For every $\ell>0$, let $\G_\ell$ be a graph consisting of one half-line and a compact core $\K_\ell$ given by a cut edge $\mathcal{B}_\ell$, of length $\ell$, joining two disjoint subgraphs $\Gamma_1,\Gamma_2$, such that $\Gamma_1\cap\h=\emptyset$ (see, e.g., Figure \ref{FIGURE-one halfline metric.}). Let also $C_{\K_\ell}$ be the optimal constant \eqref{EQ-def C_6.K} when $\G=\G_\ell$. Then,
		 \begin{equation}
		  \label{eq-asint1}
		  C_{\K_\ell}\longrightarrow C_{\rr^+},\qquad\text{as}\quad\ell\to\infty.
		 \end{equation}
	\end{prop}
	
	\begin{proof}
		First, identify, for the sake of simplicity, $\mathcal{B}_\ell$ with the interval $[0,\ell]$ and $\h$ with the interval $[\ell,\infty)$. With a little abuse of notation, let also $\{\phi_{\lambda}\}$ be the sequence of the half-solitons on $\rr^+$, i.e. the restrictions to $\rr^+$ of the functions given by \eqref{EQ-def soliton lambda}.
		
		Then, set $\lambda:=\ell^{-1/2}$, for every $\ell>0$, and define the sequence $\{u_\lambda\}\subset H^1(\G_\ell)$ such that
		\[
		u_\lambda(x):=\begin{cases}
		\phi_\lambda(x) & \text{if }x\in\mathcal{B}_\ell\cup\h\\
		\phi_\lambda(0) & \text{if }x\in\Gamma_1\\
		\phi_\lambda(\ell) & \text{if }x\in\Gamma_2\,.
		\end{cases}
		\]
		Recalling that $\phi_1(0)=1$, we have
		\[\begin{array}{l}
			\|u_\lambda\|_{2,\G_\ell}^2=\|\phi_1\|_{2,\rr^+}^2+\ell^{-1/2}\big(\,|\Gamma_1|+\phi_1^2(\ell^{1/2})|\Gamma_2|\,\big)\\[.2cm]
			\|u_\lambda'\|_{2,\G_\ell}^2=\ell^{-1}\|\phi_1'\|_{2,\rr^+}^2\\[.2cm]
			\|u_\lambda\|_{6,\K_\ell}^6=\ell^{-1}\|\phi_1\|_{L^6(0,\ell^{1/2})}^6+\ell^{-3/2}\big(\,|\Gamma_1|+\phi^6(\ell^{1/2})|\Gamma_2|\,\big),
		\end{array}
		\]
		so that, as $\ell\to\infty$,
		\[
		Q(u_\lambda)=\frac{\|\phi_1\|_{L^6(0,\ell^{1/2})}^6+\ell^{-1/2}\big(\,|\Gamma_1|+\phi^6(\ell^{1/2})|\Gamma_2|\big)}{\Big(\|\phi_1\|_{2,\rr^+}^2+\ell^{-1/2}|\Gamma_1|+\ell^{-1/2}\phi^2(\ell^{1/2})|\Gamma_2|\Big)^2\|\phi_1'\|_{2,\rr^+}^2}\longrightarrow\frac{\|\phi_1\|_{6,\rr^+}^6}{\|\phi_1\|_{2,\rr^+}^4\|\phi_1'\|_{2,\rr^+}^2}=C_{\rr^+}
		\]
		since $\phi_1$ is an optimizer of $C_{\rr^+}$. Hence, since $Q(u_\lambda)\leq C_{\K_\ell}<C_{\rr^+}$ for every $\ell>0$, \eqref{eq-asint1} follows.
	\end{proof}

	Now, we can focus on the existence of a sequence of graphs whose optimal constants converge, on the contrary, to $1$.
	
	 Recalling the proof of Proposition \ref{PROP- C_6,K geq 1 G with 1 half-line}, note that 1 is the value of $C_\K$ if one restrict the maximization in \eqref{EQ-def C_6.K} to functions that are constant on the compact core. As already pointed out in the previous section, such functions cannot be actual optimizers, but nevertheless the following proposition shows that when the compact core becomes too intricate, i.e. it has a large total length but a small diameter, then the optimizers cannot exhibit a significantly different behaviour. 

	\begin{figure}[t]
		\centering
		\includegraphics[width=0.4\textwidth]{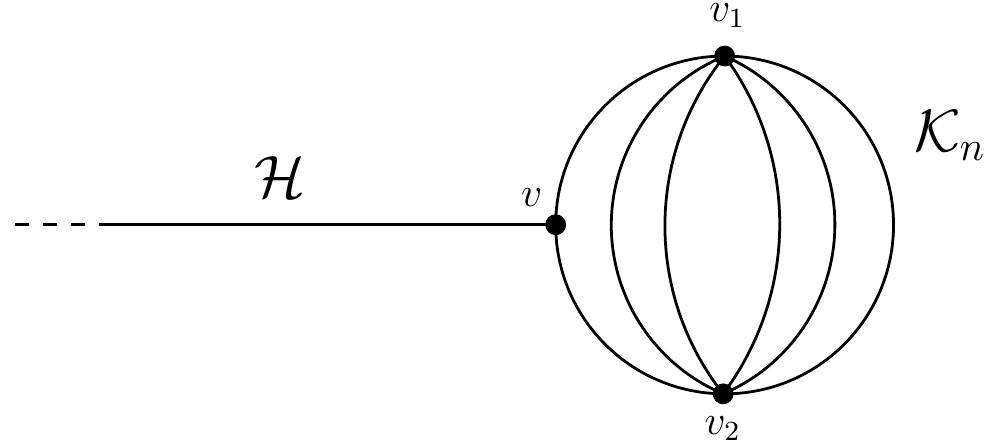}
		\caption{example of a graph $\G_n$ as in Proposition \ref{PROP-METRIC INTRICATED K}. Here, for every $n$, $\K_n$ consists of $2n+1$ edges, all of length smaller than a fixed constant, joining the vertices $v_1,\,v_2$, plus two additional edges linking $v_1$ and $v_2$ to $\h$.}
		\label{FIGURE-intricated K}
	\end{figure}
	
	\begin{prop}
		\label{PROP-METRIC INTRICATED K}
		For every $n\in\mathbb{N}$, let $\G_n$ be the graph given by a half-line $\h$ and a compact core $\K_n$ attached to the origin of $\h$ at some point $\vv$ (see, e.g, Figure \ref{FIGURE-intricated K}). Let also $C_{\K_n}$ be the optimal constant \eqref{EQ-def C_6.K} when $\G=\G_n$. If $\text{diam}(\K_n)\leq C$ uniformly in $n$, for some $C>0$, and $|\K_n|\to+\infty$ as $n\to\infty$, then 
		\begin{equation}
		 \label{eq-asint2}
		 C_{\K_n}\longrightarrow1,\qquad\text{as}\quad n\to\infty.
		\end{equation}
	\end{prop}
	
	\begin{proof}
	    We divide the proof in two steps.
	    
	    \emph{Step (i): behavior of the optimizers of $C_{\K_n}$.} By Proposition \ref{PROP- C_6,K geq 1 G with 1 half-line}, $C_{\K_n}>1>C_\rr$, and hence, for every $n\in\mathbb{N}$, Lemma \ref{LEM-C K attained} entails that there exists $u_n\in H^1(\G_n)$, $u\neq0$, such that $Q(u_n)=C_{\K_n}$. Moreover, by standard regularity theory for \eqref{eq-ELGN} (see e.g. \cite[Proposition 3.3]{AST-CVPDE}) we can assume $u_n>0$ and, from the homogeneity of $Q$, $\|u_n\|_{2,\G}^2=\mu_{\K_n}$ so that
		\begin{equation}
			\label{EQ-quotient optimal proof intricated K}
			\frac{\|u_n\|_{6,\K_n}^6}{\|u_n'\|_{2,\G_n}^2}=3,\qquad\forall n\in\mathbb{N}\,.
		\end{equation}
		On the other hand, note that, $\|u_n\|_{\infty,\G_n}=\|u_n\|_{\infty,\K_n}$. Indeed, if this is not the case, simply setting $v_n(x):=\min\{u_n(x),\|u_n\|_{\infty,\K_n}\}$ for every $x\in\G_n$, one can see that $v_n\in H^1(\G_n)$ and that $Q(u_n)<Q(v_n)$, which is a contradiction. As a consequence, denoting by $\overline{x}_n\in\K_n$ a point such that $u_n(\overline{x}_n)=\|u_n\|_{\infty,\G_n}=:M_n$, and by $\gamma_n\subset\K_n$ the smallest path from $\overline{x}_n$ to vertex $\vv$, one can easily check that $\K_n/\gamma_n$ is connected, regardless of the location of $\overline{x}_n$.
		
		Now, consider the restriction of $u_n$ to $\gamma_n\cup\h$, i.e. $u_{n_{|\gamma_n\cup\h}}$, and let $\eta_n^1:[0,+\infty)\to\rr$ be its decreasing rearrangement on $\rr^+$, so that $\eta_n^1(0)=M_n$,
		\[
		\int_{0}^{+\infty}|(\eta_n^1)'|^2\,dx\leq\int_{\gamma_n\cup\h}|u_n'|^2\,dx\qquad\text{and}\qquad\int_{0}^{+\infty}|\eta_n^1|^p\,dx=\int_{\gamma_n\cup\h}|u_n|^p\,dx\qquad\forall p\geq1.
		\]
		On the other hand, since $\K_n/\gamma_n$ is connected, the image of $u_{n_{|\overline{\K_n/\gamma_n}}}$ is connected in turn. Setting $\ell_n:=|\overline{\K_n/\gamma_n}|$, we can define, therefore, the decreasing rearrangement of $u_{n_{|\overline{\K_n/\gamma_n}}}$ on the interval $[0,\ell_n]$, i.e. $\eta_n^2:[0,\ell_n]\to\rr$, which satisfies $\eta_n^2(0)=M_n$,
		\[
		\quad\int_{0}^{\ell_n}|(\eta_n^2)'|^2\,dx\leq\int_{\K_n/\gamma_n}|u_n'|^2\,dx\qquad\text{and}\qquad\int_{0}^{\ell_n}|\eta_n^2|^p\,dx=\int_{\K_n/\gamma_n}|u_n|^p\,dx\qquad\forall p\geq1.
		\]
		Then, define the function $\eta_n:(-\infty,\ell_n]\to\rr$ such that
		\[
		\eta_n(x):=\begin{cases}
		\eta_n^1(-x) & \text{if }x<0\\[.2cm]
		\eta_n^2(x) & \text{if }x\in[0,\ell_n].
		\end{cases}
		\]
		Exploiting the properties of $\eta_n^1$ and $\eta_n^2$, one can easily see that $\eta_n\in H^1(-\infty,\ell_n)$, $\|\eta_n'\|_{L^2(-\infty,\ell_n)}\leq\|u_n'\|_{2,\G_n}$ and $\|\eta_n\|_{L^p(-\infty,\ell_n)}=\|u_n\|_{p,\G_n}$ for every $p\geq1$.
		
		As a further step, arguing as in \cite[Proof of Lemma 4.4, Step 2--Step 3]{AST-CMP}, one can construct a sequence of functions $v_n\in H^1(\rr)$ such that, for some $\theta_n:=\theta_n(u_n)\in(0,\mu_{\K_n})$
		\begin{itemize}
			\item[(a)] $v_n(0)=\eta_n(0)=M_n$;\\[-.2cm]
			\item[(b)] $\int_\rr|v_n|^2\,dx=\int_{-\infty}^{\ell_n}|\eta_n|^2\,dx-\theta_n=\mu_{\K_n}-\theta_n$;\\[-.2cm]
			\item[(c)] $\int_\rr|v_n'|^2\,dx\leq\int_{-\infty}^{\ell_n}|\eta_n'|^2\,dx+\frac{C}{\ell_n^2}\theta_n^{1/2}\leq\|u_n'\|_{2,\G_n}^2+\frac{C}{\ell_n^2}\theta_n^{1/2}$;\\[-.2cm]
			\item[(d)] $\int_\rr|v_n|^6\,dx\geq\int_{-\infty}^{\ell_n}|\eta_n|^6\,dx-\frac{C}{\ell_n^2}\theta_n=\|u_n\|_{6,\G_n}^6-\frac{C}{\ell_n^2}\theta_n$;
		\end{itemize}
		where $C>0$ does not depend on $n$.
		
		Hence, combining (a)-(d) with \eqref{EQ-GN 6} for $\G=\rr$,
		\begin{equation*}
			\|u_n\|_{6,\G_n}^6-\frac{C}{\ell_n^2}\theta_n\leq\|v_n\|_{6,\rr}^6\leq3\frac{(\mu_{\K_n}-\theta_n)^2}{\mu_\rr^2}\|v_n'\|_{2,\rr}^2\leq3\frac{(\mu_{\K_n}-\theta_n)^2}{\mu_\rr^2}\Big(\|u_n'\|_{2,\G_n}^2+\frac{C}{\ell_n^2}\theta_n^{1/2}\Big)
		\end{equation*}
		and thus, rearranging terms and recalling that $\theta_n<\mu_{\K_n}<\sqrt{3}<\mu_\rr$ (possibly redefining $C$), 
		\begin{equation}
			\label{EQ-bound on L6 proof intricated}
			\|u_n\|_{6,\G_n}^6\leq3\frac{\mu_{\K_n}^2}{\mu_\rr^2}\|u_n'\|_{2,\G_n}^2+\frac{C}{\ell_n^2}\,.
		\end{equation}
		Moreover, plugging \eqref{EQ-quotient optimal proof intricated K} into \eqref{EQ-bound on L6 proof intricated}, we get
		\begin{equation*}
			\|u_n'\|_{2,\G_n}^2=\frac{1}{3}\|u_n\|_{6,\K_n}^6\leq\frac{1}{3}\|u_n\|_{6,\G_n}^6\leq\frac{\mu_{\K_n}^2}{\mu_\rr^2}\|u_n'\|_{2,\G_n}^2+\frac{C}{\ell_n^2},
		\end{equation*}
		that is
		\begin{equation}
		\label{eq-boundprel}
			\Big(1-\frac{\mu_{\K_n}^2}{\mu_\rr^2}\Big)\|u_n'\|_{2,\G_n}^2\leq\frac{C}{\ell_n^2}\,.
		\end{equation}
		Since both $|\K_n|\to+\infty$  as $n\to\infty$ and $|\gamma_n|\leq\text{diam}(\K_n)\leq C$ for every $n$, then $\ell_n\sim|\K_n|$ provided $n$ is large enough, so that \eqref{eq-boundprel} implies 
		\begin{equation}
		\label{EQ-bound on u' proof intricated}
		\|u_n'\|_{2,\G_n}|\K_n|\leq C
		\end{equation}
		 
		\noindent uniformly in $n$.
		 
		\emph{Step (ii): proof of \eqref{eq-asint2}.} Let $m_n:=\min_{x\in \K_n}u_n(x)$ and $y_n\in\K_n$ be such that $u_n(\overline{y}_n)=m_n$. Furthermore, let $z_n\in\h$ be the closest point (of the half-line) to the compact core such that $u_n(z_n)=m_n$ (possibly $z_n=\vv$ if the minimum is attained at the vertex joining $\h$ and $\K_n$).
		
		We then consider the functions defined as
		
		\[
		\widetilde{u}_n(x):=\begin{cases}
		m_n & \text{if }x\in\K_n\\
		u_n(x+z_n) & \text{if }x\in\h\,.
		\end{cases}
		\]
		
		\noindent Clearly, $\widetilde{u}_n\in H^1(\G_n)$ and, by construction,  $\|\widetilde{u}\|_{2,\G_n}\leq\|u_n\|_{2,\G_n}$, $\|\widetilde{u}_n'\|_{2,\G_n}\leq\|u_n'\|_{2,\G_n}$, so that 
		
		\begin{equation}
		\label{eq-ckaux}
			C_{\K_n}\leq\frac{\|u_n\|_{\infty,\G_n}^6|\K_n|}{\|u_n\|_{2,\G_n}^4\|u_n\|_{2,\G_n}^2}=\frac{\|u_n\|_{\infty,\G_n}^6}{m_n^6}\frac{\|\wt{u}_n\|_{6,\K_n}^6}{\|\wt{u}_n\|_{2,\G_n}^4\|\wt{u}_n'\|_{2,\G_n}^2}\leq\frac{\|u_n\|_{\infty,\G_n}^6}{m_n^6},
		\end{equation}
		
		\noindent the last inequality being motivated by Remark \ref{rem-cost+roba}.
		 
		 As $C_{\K_n}>1$, it is then left to estimate $\|u_n\|_{\infty,\G_n}^6/m_n^6$. First, recall that $\|u_n\|_{\infty,\G_n}=\|u_n\|_{\infty,\K_n}=u_n(\overline{x}_n)$, with $\overline{x}_n\in\K_n$ defined in Step (i), whereas $m_n=u_n(\overline{y}_n)$. Let  $\Gamma_n\subset\K_n$ be the smallest path from $\overline{y}_n$ to $\overline{x}_n$. Then we have
		\begin{equation}
			\label{EQ-quotient rewritten}
			\frac{m_n}{\|u_n\|_{\infty,\G_n}}=1-\frac{u_n(\bar{x}_n)-u_n(\bar{y}_n)}{\|u_n\|_{\infty,\G_n}}=1-\frac{\int_{\Gamma_n}u_n'\,dx}{\|u_n\|_{\infty,\G_n}}\,.
		\end{equation}
		Now, let us show that 
		\begin{equation}
			\label{EQ-final quotient to 0}
			\lim_{n}\frac{\int_{\Gamma_n}u_n'\,dx}{\|u_n\|_{\infty,\G_n}}=0\,.
		\end{equation}
		By H\"older's inequality
		\[
		\int_{\Gamma_n}|u_n'|\,dx\leq|\Gamma_n|^{1/2}\|u_n'\|_{2,\Gamma_n}\leq\text{diam}(\K_n)^{1/2}\|u_n'\|_{2,\G_n},
		\]
		so that
		\begin{equation}
			\label{EQ-quotient u' u infty}
			\frac{\int_{\Gamma_n}u_n'\,dx}{\|u_n\|_{\infty,\G_n}}\leq\frac{\text{diam}(\K_n)^{1/2}\|u_n'\|_{2,\G_n}}{\|u_n\|_{\infty,\G_n}}\,.
		\end{equation}
		Moreover, by \eqref{EQ-quotient optimal proof intricated K},
		\[
		\|u_n'\|_{2,\G_n}^2=\frac{1}{3}\|u_n\|_{6,\K_n}^6\leq\frac{1}{3}\|u_n\|_{\infty,\G_n}^6|\K_n|,
		\]
		which yields at
		\begin{equation}
			\label{EQ-u' over u infty bounded above}
			\frac{\|u_n'\|_{2,\G_n}}{\|u_n\|_{\infty,\G_n}}\leq\frac{|\K_n|^{1/2}}{\sqrt{3}}\|u_n\|_{\infty,\G_n}^2\,.
		\end{equation}
		By \eqref{EQ-GN infty} and the fact that the optimal constant $C_\infty=\sqrt{2}$, for every $\G_n$ (see Section \ref{sec-prel})
		\[
		\frac{|\K_n|^{1/2}}{\sqrt{3}}\|u_n\|_{\infty,\G_n}^2\leq\frac{2|\K_n|^{1/2}}{\sqrt{3}}\mu_{\K_n}^{1/2}\|u_n'\|_{2,\G_n}\leq \frac{C}{|\K_n|^{1/2}}\longrightarrow0\qquad\text{as}\quad n\to\infty,
		\]
		making use of \eqref{EQ-bound on u' proof intricated}. Plugging into \eqref{EQ-u' over u infty bounded above}, we get
		\[
		\lim_{n}\frac{\|u_n'\|_{2,\G_n}}{\|u_n\|_{\infty,\G_n}}=0,
		\]
		which, combined with \eqref{EQ-quotient u' u infty} and $\text{diam}(\K_n)< C$, implies \eqref{EQ-final quotient to 0}. Hence, passing to the limit in \eqref{EQ-quotient rewritten},
		\[
		\lim_{n}\frac{m_n}{\|u_n\|_{\infty,\G_n}}=1
		\]
		and, recalling \eqref{eq-ckaux} and the fact that $C_{\K_n}>1$, \eqref{eq-asint2} follows.
	\end{proof}

	The proof of Theorem \ref{THM2} is therefore a straightforward application of Propositions \ref{PROP-METRIC TADPOLE} and \ref{PROP-METRIC INTRICATED K}.
	
	\begin{proof}[Proof of Theorem \ref{THM2}] Fix $\varepsilon>0$. The existence of $\G_\varepsilon^1$ is immediately guaranteed by Proposition \ref{PROP-METRIC TADPOLE}. At the same time, considering graphs as in Proposition \ref{PROP-METRIC INTRICATED K}, we easily obtain $\G_\varepsilon^2$.
	\end{proof}
	

	\subsection{Graphs without terminal edges and cycle-coverings, with at least two half-lines}
	\label{sec-iv}
	
	In this last section, we discuss graphs without terminal edges and cycle-coverings, with at least two half-lines (for which it is true that $\mu_{\rr^+}<\mu_\K\leq\mu_\rr$, i.e. $C_\rr\leq C_\K < C_{\rr^+}$, and that ground states exist if and only if $\mu\in[\mu_\K,\mu_\rr]$, provided that $\mu_\K\neq\mu_\rr$).
	
	Recall that, in a graph having at least two half-lines and neither a terminal edge nor a cycle-covering, there exists of at least one cut edge in the compact core, that is an edge that provides the only connection between two disjoint subgraphs. For the sake of simplicity, in what follows we develop our analysis in the case of a signpost graph, as in Figure \ref{FIGURE-signpost}, with exactly one cut edge and a circle attached to it. 
	
	For every $\ell>0$, let $\G_\ell$ be the graph depicted in Figure \ref{FIGURE-signpost}, with compact core $\K_\ell$ given by the cut edge $\mathcal{B}_\ell$, of length $\ell$, and the circle $\Gamma$. Moreover, denote by $\mathcal{H}_1,\mathcal{H}_2$ the left and the right half-line of $\G_\ell$, respectively, both identified with $[\ell,+\infty)$. In addition, identify the cut edge $\mathcal{B}_\ell$ with the interval $[0,\ell]$, in such a way that $x_{\mathcal{B}_\ell}=\ell$ ($x_{\mathcal{B}_\ell}$ being the coordinate on $\mathcal{B}_\ell$) denotes the vertex $\vv$ such that $\{\vv\}=\h_1\cap\h_2$. Let also $C_{\K_\ell}$ be the optimal constant \eqref{EQ-def C_6.K} when $\G=\G_\ell$.
	
	We then have the first asymptotic result.
	
	\begin{figure}[t]
		\centering
		\includegraphics[width=0.6\textwidth]{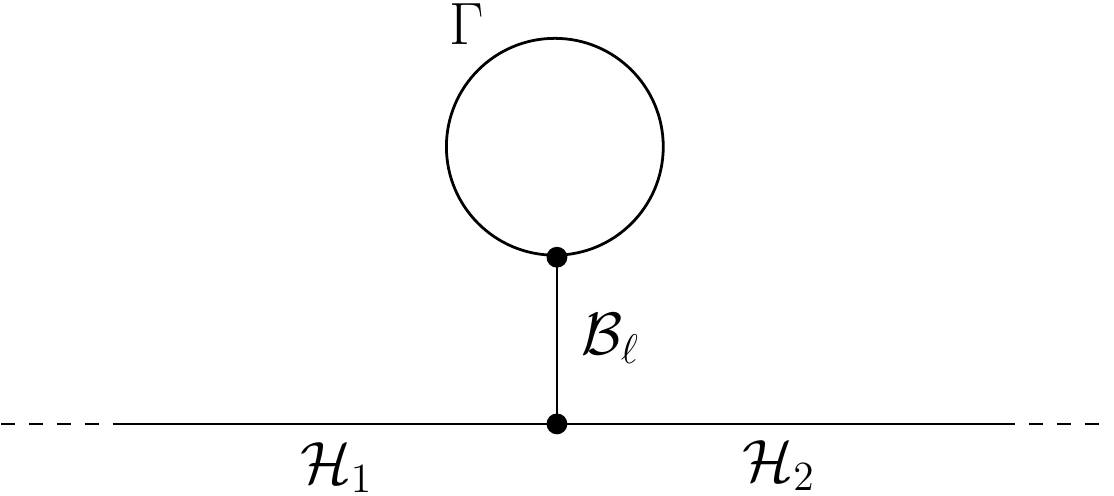}
		\caption{the signpost graph.}
		\label{FIGURE-signpost}
	\end{figure}
	
	\begin{prop}
		\label{PROP-METRIC SIGNPOST 1}
         Let $\G_\ell$ be the graph depicted in Figure \ref{FIGURE-signpost}. Then, the following asymptotics holds:
        \begin{equation}
         \label{eq-asint3}
         C_{\K_\ell}\longrightarrow C_{\rr^+}\qquad\text{as}\quad \ell\to\infty.
        \end{equation}
	\end{prop}

	\begin{proof}
	    As in the proof of Proposition \ref{PROP-METRIC TADPOLE}, let $\{\phi_{\lambda}\}$ be the sequence of the half-solitons on $\rr^+$, i.e. the sequence of the restrictions to $\rr^+$ of the functions given by \eqref{EQ-def soliton lambda}. Then, set $\lambda:=\ell^{-1/2}$, for every $\ell>0$, and define the sequence $\{u_\lambda\}\subset H^1(\G_\ell)$ such that
		\[
		u_\lambda(x):=\begin{cases}
		\phi_\lambda(0) & \text{if }x\in\Gamma\\
		\phi_\lambda(x) & \text{if }x\in\G_\ell/\Gamma.
		\end{cases}
		\]
		Note that, as both $\h_1$ and $\h_2$ are identified with the interval $[\ell,\infty)$, then $u_{\lambda_{|\h_1}}\equiv u_{\lambda_{|\h_2}}$. Hence, direct computations (recalling $\phi_1(0)=1$) yield
		\[
		\begin{split}
		\|u_\lambda\|_{2,\G_\ell}^2=&\lambda|\Gamma|+\int_{\mathcal{B}_\ell\cup\h_1}\lambda|\phi_1(\lambda x)|^2\,dx+\int_{\h_2}\lambda|\phi_1(\lambda x)|^2\,dx=\ell^{-1/2}|\Gamma|+\|\phi_1\|_{2,\rr^+}^2+\|\phi_1\|_{L^2(\ell^{1/2},\infty)}^2\\[.2cm]
		\|u_\lambda'\|_{2,\G_\ell}^2=&\int_{\mathcal{B}_\ell\cup\h_1}\lambda^3|\phi_1'(\lambda x)|^2\,dx+\int_{\h_2}\lambda^3|\phi_1'(\lambda x)|^2\,dx=\ell^{-1}\Big(\|\phi_1'\|_{2,\rr^+}^2+\|\phi_1'\|_{L^2(\ell^{1/2},\infty)}^2\Big)\\[.2cm]
		\|u_\lambda\|_{6,\K_\ell}^6=&\lambda^3|\Gamma|+\int_{\mathcal{B}_\ell}\lambda^3|\phi_1(\lambda x)|^6\,dx=\ell^{-1}\Big(\ell^{-1/2}|\Gamma|+\|\phi_1\|_{L^6(0,\ell^{1/2})}^6\Big)
		\end{split}
		\]
		and, since as $\ell\to\infty$, $\ell^{-1/2}|\Gamma|, \|\phi_1\|_{L^2(\ell^{1/2},\infty)}$ and $\|\phi_1'\|_{L^2(\ell^{1/2},\infty)}$ tend to $0$, and $\|\phi_1\|_{L^6(0,\ell^{1/2})}\to\|\phi\|_{6,\rr^+}$, it follows that
		\[
		Q(u_\lambda)\longrightarrow\frac{\|\phi_1\|_{6,\rr^+}^6}{\|\phi_1\|_{2,\rr^+}^4\|\phi_1'\|_{2,\rr^+}^2}=C_{\rr^+}\,.
		\]
		Combining with the fact that $Q(u_\lambda)\leq C_\K<C_{\rr^+}$ for every $\lambda$, we obtain \eqref{eq-asint3}.
	\end{proof}
	
	\begin{rem}
	 \label{rem-ex}
	 Clearly, the proof of Proposition \ref{PROP-METRIC SIGNPOST 1} provides an example of graph such that $C_\K\neq C_\rr$, showing, as anticipated in the Introduction, that the class of graphs of case (iv) in Theorem \ref{THM1} fulfilling $\mu_\K\neq\mu_\rr$ is not empty.
	\end{rem}

	\begin{rem}
		It is readily seen that the proof of Proposition \ref{PROP-METRIC SIGNPOST 1} can be straightforwardly generalized to the case of several cut edges joining components (possibly) different than circles, simply repeating the previous construction for any fixed cut edge.
	\end{rem}
	
	The existence of a sequence of graphs whose optimal constants converge to $C_{\rr}$, on the contrary, requires more efforts. We aim at obtaining such a result by discussing the behaviour of $\G_\ell$, as before, but now in the regime $\ell\to0$.
	
	Some preliminary steps are required: a characterization of the behavior of the (possible) optimizers of $C_{\K_\ell}$ (Lemmas \ref{LEM- max optimizer on bridge}--\ref{LEMMA-optimizer on half-lines}) and a monotonicity property of $C_{\K_\ell}$ with respect to $\ell$ (Lemma \ref{LEMMA-monotonicity C_6 in G_l}). For the sake of simplicity, we denote in the following by $\h$ the union of the half-lines of $\G_\ell$, i.e. $\h:=\h_1\cup\h_2$.
	
	\begin{lem}
		\label{LEM- max optimizer on bridge}
		Let $\G_\ell$ be the graph depicted in Figure \ref{FIGURE-signpost}. If $u\in H^1(\G_\ell)$, $u>0$, is an optimizer of $C_{\K_\ell}$, then
		\[
		 M_{\mathcal{B}_\ell}:=\max_{x\in\mathcal{B}_\ell}u(x)>\max_{x\in\h}u(x)=:M_\h.
		\]
\end{lem}
		
		\begin{proof}
		Suppose, by contradiction, that $M_{\mathcal{B}_\ell}\leq M_\h$. Then, for a.e. $t$ in the image of $u$,
		\[
		\symbol{35}\{x\in\G_\ell\,:\,u(x)=t\}\geq2\,.
		\]
		As a consequence, with $\widehat{u}$ being the symmetric rearrangement of $u$,
		\begin{equation}
		\label{EQ-contradiction lemma symmetry 1}
		C_{\K_\ell}=Q(u)\leq\frac{\|u\|_{6,\G_\ell}^6}{\|u\|_{2,\G_\ell}^4\|u'\|_{2,\G_\ell}^2}\leq\frac{\|\widehat{u}\|_{6,\rr}^6}{\|\widehat{u}\|_{2,\rr}^4\|\widehat{u}'\|_{2,\rr}^2}\leq C_\rr\,.
		\end{equation}
		Now, if $C_{\K_\ell}\neq C_\rr$, then it directly entails a contradiction, in view of Proposition \ref{PROP-C_6,K between rr and rr+ }. On the other hand, if $C_{\K_\ell}=C_\rr$, then \eqref{EQ-contradiction lemma symmetry 1} implies that $\widehat{u}$ is an optimizer of $C_\rr$ on the real line and, by the properties of symmetric rearrangements, that
		\begin{equation}
		 \label{eq-symnon}
		 \|\widehat{u}'\|_{2,\rr}=\|u'\|_{2,\G_\ell}.
		\end{equation}
        However, since $u$ runs through a vertex of degree $3$ and $M_{\mathcal{B}_\ell}\leq M_\h$, there exists a subregion of $\G_\ell$ of positive measure where all the values attained by $u$ possess at least three pre-images ,which contradicts \eqref{eq-symnon}.
	\end{proof}
	
	\begin{lem}
		\label{LEMMA-optimizer on half-lines}
		Let $\G_\ell$ be the graph depicted in Figure \ref{FIGURE-signpost}. If $u\in H^1(\G_\ell)$, $u>0$, is an optimizer of $C_{\K_\ell}$, then $u_{|_\h}$ is symmetric with respect to $\vv$ (recall $\{\vv\}:=\h_1\cap\h_2$) and non-increasing both on $\h_1$ and on $\h_2$.
	\end{lem}

	\begin{proof}
	Let us prove first that $u_{|\h}$ attains the maximum $M_\h$ at the vertex $\vv$ and that it is non-increasing both on $\h_1$ and on $\h_2$. To this aim, assume by contradiction that $u_{|\h}$ does not posses any of the previous features. 
	
	Let $x_0$ be the closest point to the circle $\Gamma$ in $\mathcal{B}_\ell$ such that $u(x_0)=M_{\mathcal{B}_\ell}$ (see, e.g., Figure \ref{FIGURE-proof of lemma}(a)). Denote, also, by $\overline{G}$ the subgraph of $\G_\ell$ obtained by removing the circle $\Gamma$ and the segment that joins $x_0$ and the vertex $\vv'$ given by the intersection between $\Gamma$ and $\mathcal{B}_\ell$. Then, let
	\[
	\begin{split}
	A:=&\{x\in\overline{G}\,:\, u(x)>M_\h\}\\
	B:=&\{x\in\overline{G}\,:\, u(x)\leq M_\h\}\,.
	\end{split}
	\]
	and let $u_A,u_B$ be the restrictions of $u$ to $A$ and $B$, respectively. Note that $A$ and $B$ do not need to be connected sets, but the images of $u_A$ and $u_B$ are nevertheless both connected.
	
	By Lemma \ref{LEM- max optimizer on bridge} $M_\h< M_{\mathcal{B}_\ell}=\max_{x\in A}u(x)$, and hence $A\subset\mathcal{B}_\ell$ and $x_0+|A|\leq\ell$. Therefore, letting $u_A^*$ be the decreasing rearrangement of $u_A$ on $[0,|A|]$ (see, e.g., Figure \ref{FIGURE-proof of lemma}(b)), we get $u_A^*(0)=M_{\mathcal{B}_\ell}$, $u_A^*(|A|)=M_\h$ and $\|(u_A^*)'\|_{L^2(0,|A|)}\leq\|u_A'\|_{L^2(A)}$. On the other hand, since $M_\h=\max_{x\in B}u(x)$, it follows that $\symbol{35}\{x\in B\,:\,u(x)=t\}\geq2$ for almost every value $t\in(0,M_\h]$, and rearranging symmetrically $u_B$ on $\rr$ (see again Figure \ref{FIGURE-proof of lemma}(b)) we get $\widehat{u}_B(0)=M_\h$ and $\|\widehat{u}_B'\|_{L^2(\rr)}\leq\|u_B'\|_{L^2(B)}$. In addition, one can see that if $u_{|\h}$ does not attain $M_\h$ at $\vv$ or is non-monotone either on $\h_1$ or on $\h_2$, then there exists a subregion of $B$ with positive measure such that all the values attained there by $u_{\mid B}$ are actually realized at least three times each, leading to
	\begin{equation}
	 \label{eq-less}
	 \|\widehat{u}_B'\|_{L^2(\rr)}<\|u_B'\|_{L^2(B)}
	\end{equation}

	Then, we can use $u_A^*,\widehat{u}_B$ in order to construct a new function $\wt{u}\in H^1(\G_{\ell})$ (see, e.g., Figure \ref{FIGURE-proof of lemma}(c)). First, on $\G_\ell/\overline{G}$, set $\wt{u}\equiv u$, and on $\mathcal{B}_\ell\cap[x_0,x_0+|A|]$, set $\wt{u}(x)=u_A^*(x-x_0)$. In addition, consider the restriction of $\widehat{u}_B$ to the interval $\big[-\frac{\ell-x_0-|A|}{2},\frac{\ell-x_0-|A|}{2}\big]$ (see again Figure \ref{FIGURE-proof of lemma}(b)), rearrange it decreasingly on $[0,\ell-x_0-|A|]$, denoting by $\widehat{u}_B^*$ such a rearrangement, and then set $\wt{u}(x)=\widehat{u}_B^*(x-x_0-|A|)$ on $\mathcal{B}_\ell\backslash[0,x_0+|A|]$ (assuming that on $\mathcal{B}_\ell$ the vertex $\vv'$ is represented by $x=0$). Finally, define $\wt{u}$ on $\h$ as the restriction of $\widehat{u}_B$ to $\rr/\big[-\frac{\ell-x_0-|A|}{2},\frac{\ell-x_0-|A|}{2}\big]$, glued together at $0$.
	
	\begin{figure}[t]
		\centering
		\subfloat[][the function $u$ on $\G_{\ell}$]
		{\includegraphics[width=.35\columnwidth]{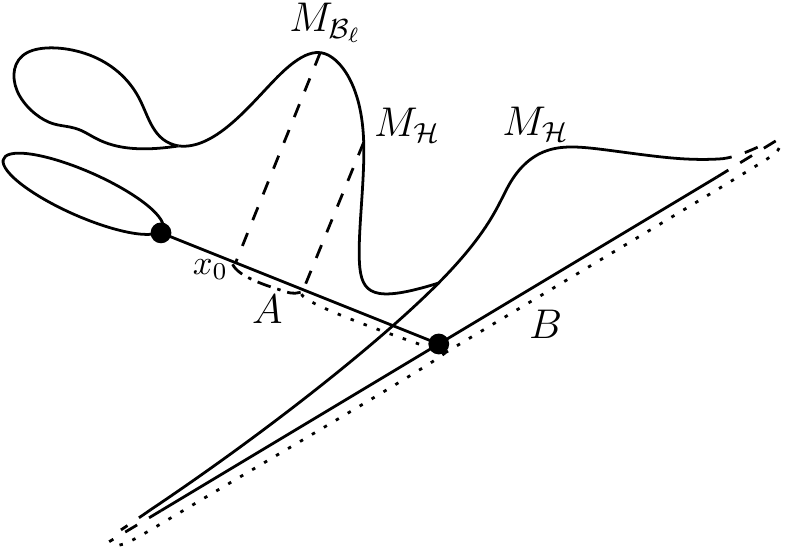}} \qquad\qquad
		\subfloat[][the decreasing rearrangement of $u_A$ on \text{$[0,|A|]$} and the symmetric rearrangement of $u_B$ on $\rr$]
		{\includegraphics[width=.50\columnwidth]{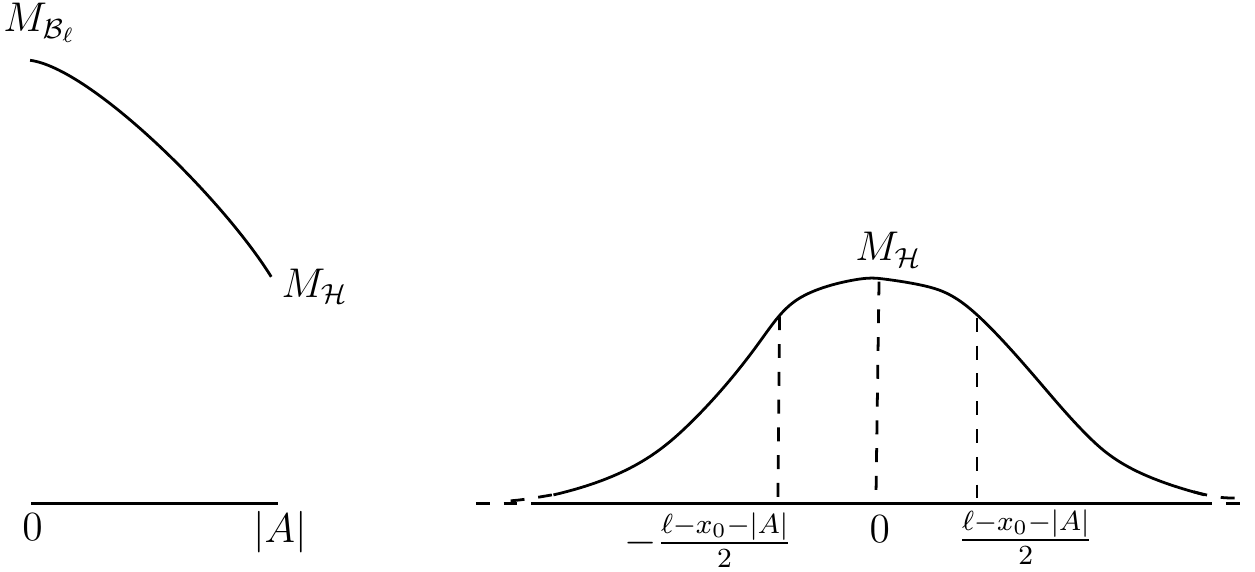}}\qquad
		
		\medskip
		\medskip
		
		\subfloat[][the function $\tilde{u}$ on $\G_{\ell}$]
		{\includegraphics[width=.35\textwidth]{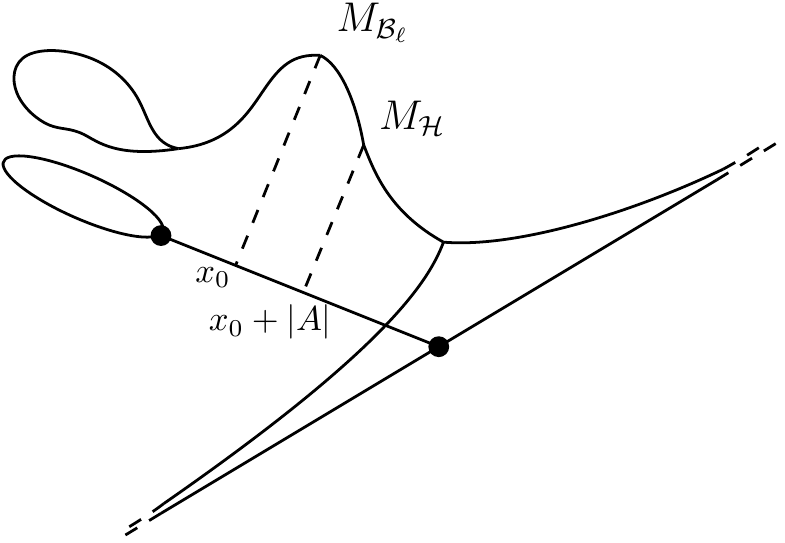}}
		\caption{the steps of the construction of $\tilde{u}$ starting from $u$ as in the proof of Lemma \ref{LEMMA-optimizer on half-lines}.}
		\label{FIGURE-proof of lemma}
	\end{figure}
	
	As a consequence, we have
	\[
	\begin{array}{l}
	\|\wt{u}\|_{2,\G_{\ell}}=\|u\|_{2,\G_{\ell}}\\[.2cm]
	\|\wt{u}'\|_{2,\G_{\ell}}<\|u'\|_{2,\G_{\ell}}\\[.2cm]
	\|\wt{u}\|_{6,\K_{\ell}}\geq\|u\|_{6,\K_{\ell}},
	\end{array}
	\]
	the strict inequality being given by \eqref{eq-less}, and thus
	\begin{equation*}
	C_{\K_\ell}=Q(u)< Q(\wt{u})\leq C_{\K_\ell},
	\end{equation*}
	i.e. a contradiction. Hence, $u_{|\h}$ has a maximum at $\vv$ and is non-increasing both on $\h_1$ and on $\h_2$.
	
	Moreover, if $u_{\h}$ is non-symmetric with respect to $\vv$, then a symmetric rearrangement on $\h$ would provide a better competitor (it is well known, indeed, that on the real line a symmetric rearrangement of a non-symmetric function has a strictly smaller kinetic energy), serving again a contradiction.
	\end{proof}
	
	\begin{rem}
	 Note that neither Lemma \ref{LEM- max optimizer on bridge} nor Lemma \ref{LEMMA-optimizer on half-lines} discusses the existence of optimizers for $C_{\K_\ell}$ in general. They only state some a priori conditions to be satisfied by possible minimizers.
	\end{rem}

	\begin{lem}
		\label{LEMMA-monotonicity C_6 in G_l} 
		Let $\G_\ell$ be the graph depicted in Figure \ref{FIGURE-signpost}. Then, for every $0<\ell_1\leq\ell_2$, it is true that
		\begin{equation*}
			C_{\K_{\ell_1}}\leq C_{\K_{\ell_2}}.
		\end{equation*}
	\end{lem}
	
	\begin{proof}
		If $C_{\K_{\ell_1}}=C_\rr$, then the claim is trivial by \eqref{EQ-C_6,K between rr and rr+ }.  Let thus $C_{\K_{\ell_1}}>C_\rr$. In view of Proposition \ref{PROP-C_6,K with respect to rr+}, there exists $u\in H^1(\G_{\ell_1})$ such that $Q(u)=C_{\K}(\G_{\ell_1})$, and, by Lemma \ref{LEMMA-optimizer on half-lines}, it is symmetric on $\h$ and non-increasing both on $\h_1$ and on $\h_2$.
		
		Setting $\Lambda:=\ell_2-\ell_1$, let $I\subset\h$ be a symmetric interval of measure $|I|=\Lambda$ centered at $\vv$, $u_{|I}$ the restriction of $u$ to $I$ and $u_{|I}^*$ its decreasing rearrangement of $[0,\Lambda]$. Subsequently, consider the function $v\in H^1(\G_{\ell_2})$ defined by
		\[
		v(x):=\begin{cases}
		u(x) & \text{if }x\in\Gamma\cup(\mathcal{B}_{\ell_2}\cap[0,\ell_1])\\
		u_{|I}^*(x-\ell_1) & \text{if }x\in\mathcal{B}_{\ell_2}\cap[\ell_1,\ell_2]\\
		u(|x|+\Lambda/2) & \text{if }x\in\h.
		\end{cases}
		\]
		As a consequence
		\[
		\|v\|_{2,\G_{\ell_2}}=\|u\|_{2,\G_{\ell_1}}\,,\quad\|v'\|_{2,\G_{\ell_2}}<\|u\|_{2,\G_{\ell_1}}\qquad\text{and}\qquad\|v\|_{6,\K_{\ell_2}}>\|u\|_{6,\K_{\ell_1}}.
		\]
		yielding
		\[
		C_{\K_{\ell_1}}=Q(u)<Q(v)\leq C_{\K_{\ell_2}}.
		\]
	\end{proof}
	
	Now, we can prove the existence of a sequence of graphs whose optimal constants converge to $C_\rr$.
		
	\begin{prop}
		\label{PROP-METRIC SIGNPOST 2}
		Let $\G_\ell$ be the graph depicted in Figure \ref{FIGURE-signpost}. Then, the following asymptotics holds:
		\[
		 C_{\K_\ell}\longrightarrow C_\rr\qquad\text{as}\quad \ell\to0.
		\]
	\end{prop}
	
	\begin{proof}
		If $C_{\K_{\ell}}=C_\rr$, for some $\ell>0$, then the statement follows by Lemma \ref{LEMMA-monotonicity C_6 in G_l}. Assume, then, $C_{\K_\ell}>C_\rr$, for every $\ell>0$. By Proposition \ref{PROP-C_6,K with respect to rr+}, there exists a sequence $\{v_\ell\}_{\ell>0}$ such that $v_\ell\in H^1(\G_\ell)$, $v_\ell>0$ and $Q(v_\ell)=C_{\K}(\G_\ell)$. Also, by homogeneity of $Q$, we can set $\|v_\ell\|_{2,\G_\ell}^2=\mu_{\K_\ell}$, so that
		\begin{equation}
			\label{EQ-optimal quotient =3}
			\frac{\|v_\ell\|_{6,\K_\ell}^6}{\|v_\ell'\|_{2,\G_\ell}^2}=3\,\qquad\forall\ell>0\,.
		\end{equation}
		Combining \eqref{EQ-optimal quotient =3} with the modified Gagliardo-Nirenberg inequality \eqref{EQ-modified GN} leads to
		\begin{equation*}
			3=\frac{\|v_\ell\|_{6,\K_\ell}^6}{\|v_\ell'\|_{2,\G_\ell}^2}\leq\frac{\|v_\ell\|_{6,\G_\ell}^6}{\|v_\ell'\|_{2,\G_\ell}^2}\leq\frac{3\Big(\frac{\mu_{\K_\ell}}{\mu_\rr}\Big)^2\|v_\ell'\|_{2,\G_\ell}^2+C\sqrt{\mu_\rr}}{\|v_\ell'\|_{2,\G_\ell}^2}=3\left(\frac{\mu_{\K_\ell}}{\mu_\rr}\right)^2+\frac{C\sqrt{\mu_\rr}}{\|v_\ell'\|_{2,\G_\ell}^2}
		\end{equation*}
		and, rearranging terms, we get
		\begin{equation}
			\label{EQ-bound on v ell'}
			3\left(1-\frac{\mu_{\K_\ell}^2}{\mu_\rr^2}\right)\|v_\ell'\|_{2,\G_\ell}^2\leq C\sqrt{\mu_\rr}\,.
		\end{equation}
		As a consequence, if $\limsup_{\ell\to0}\|v_\ell'\|_{2,\G_\ell}=\infty$, then \eqref{EQ-bound on v ell'} immediately implies $\mu_{\K_\ell}\to\mu_\rr$, that is $C_{\K_\ell}\to C_\rr$.
		
		On the other hand, assume that $\|v_\ell'\|_{2,\G_\ell}$ is bounded. Define a new graph $\wt{\G}_\ell$ with compact core $\wt{\K}_\ell$ as follows: for every $\ell>0$, replace the cut edge $\mathcal{B}_\ell$ of $\G_\ell$ with two edges, say $\mathcal{B}_\ell^1,\mathcal{B}_\ell^2$, joining the same vertices, each of length $2\ell$. Furthermore, let $\wt{v}_\ell$ be the function in $H^1(\wt{\G}_\ell)$ defined by
		\begin{equation*}
			\wt{v}_\ell(x):=\begin{cases}
			v_\ell(x) & \text{if }x\notin\mathcal{B}_\ell^1\cup\mathcal{B}_\ell^2\\[.2cm]
			v_\ell(\frac{x}{2}) & \text{if }x\in\mathcal{B}_\ell^i,\quad i=1,2\,.
			\end{cases}
		\end{equation*}
		Now,
		\begin{equation}
		\label{EQ- norm p v tilde}
		\begin{split}
			\|\wt{v}_\ell\|_{p,\wt{\K}_\ell}^p=&\int_{\Gamma_\ell}|\wt{v}_\ell|^p\,dx+\sum_{i=1}^2\int_{\mathcal{B}_\ell^i}|\wt{v}_\ell|^p\,dx=\int_{\Gamma_\ell}|v_\ell|^p\,dx+2\int_{0}^{2\ell}|\wt{v}_\ell(x)|^p\,dx\\
			=&\int_{\Gamma_\ell}|v_\ell|^p\,dx+2\int_{0}^{2\ell}|v_\ell(x/2)|^p\,dx=\|v_\ell\|_{p,\K_\ell}^p+3\int_{\mathcal{B}_\ell}|v_\ell|^p\,dx\,,
		\end{split}
		\end{equation}
		and
		\begin{equation}
			\label{EQ- norm of v tilde '}
			\begin{split}
			\|\wt{v}_\ell'\|_{2,\wt{\G}_\ell}^2=&\int_{\wt{G}_\ell\backslash\mathcal{B}_\ell^1\cup\mathcal{B}_\ell^2}|\wt{v}_\ell'|^2\,dx+\sum_{i=1}^{2}\int_{\mathcal{B}_\ell^i}|\wt{v}_\ell'|^2\,dx=\int_{\G_\ell\backslash\mathcal{B}_\ell}|v_\ell'|^2\,dx+2\int_{0}^{2\ell}|\wt{v}_\ell'(x)|^2\,dx\\
			=&\int_{\G_\ell\backslash\mathcal{B}_\ell}|v_\ell'|^2\,dx+\frac{1}{2}\int_{0}^{2\ell}|v_\ell'(x/2)|^2\,dx=\|v_\ell'\|_{2,\G_\ell}^2\,.	
			\end{split}
		\end{equation}
		Since by construction $\wt{\G}_\ell$ admits a cycle-covering, then $C_{\wt{\K}_\ell}=C_\rr$ by Proposition \ref{PROP-C_6,K with respect to rr+}, which, combined with \eqref{EQ- norm p v tilde} and \eqref{EQ- norm of v tilde '}, entails that
		\begin{equation}
			\label{EQ-GN after double bridge}
			C_\rr\geq Q(\tilde{v}_\ell) =\frac{\|v_\ell\|_{6,\G_\ell}^6+3\int_{\mathcal{B}_\ell}|v_\ell|^6\,dx}{\Big(\|v_\ell\|_{2,\G_\ell}^2+3\int_{\mathcal{B}_\ell}|v_\ell|^2\,dx\Big)^2\|v_\ell'\|_{2,\G_\ell}^2}\,.
		\end{equation}
		In addition, given that $\|v_\ell'\|_{2,\G_\ell}$ is bounded by assumption, \eqref{EQ-GN infty} and \eqref{EQ-bound mu tilde resto del mondo} ensure that
		\begin{equation*}
			\|v_\ell\|_{\infty,\G_\ell}\leq \sqrt{2}\|v_\ell\|_{2,\G_\ell}^{1/2}\|v_\ell'\|_{2,\G_\ell}^ {1/2}\leq C\,,\qquad\forall\ell>0,
		\end{equation*}
		so that, as $\ell\to0$,
		\begin{equation*}
			\begin{split}
			\int_{\mathcal{B_\ell}}|v_\ell|^6\,dx&\leq\|v_\ell\|_{\infty,\G_\ell}^6\ell\to0\\
			\int_{\mathcal{B_\ell}}|v_\ell|^2\,dx&\leq\|v_\ell\|_{\infty,\G_\ell}^2\ell\to0\,.
			\end{split}
		\end{equation*}
		Hence, passing to the limit in \eqref{EQ-GN after double bridge},
		\[
			C_\rr\geq\limsup_{\ell\to0}\frac{\|v_\ell\|_{6,\K_\ell}^6}{\|v_\ell\|_{2,\G_\ell}^4\|v_\ell'\|_{2,\G_\ell}^2}=\limsup_{\ell\to0}C_{\K_\ell}\geq\liminf_{\ell\to0}C_{\K_\ell}\geq C_\rr
		\]
		thus concluding the proof.
	\end{proof}
	
	\begin{rem}
	 Proposition \ref{PROP-METRIC SIGNPOST 2} clearly holds as well replacing the circle $\Gamma$ with a generic compact graph admitting a cycle-covering.
	\end{rem}

	\begin{proof}[Proof of Theorem \ref{THM3}]
	 Given $\varepsilon>0$, the existence of $\G_\varepsilon^1$ and $\G_\varepsilon^2$ is a direct consequence of Propositions \ref{PROP-METRIC SIGNPOST 1} and \ref{PROP-METRIC SIGNPOST 2}.
	\end{proof}


\begin{thebibliography}{99}

\bibitem{ADST}
R. Adami, S.Dovetta, E. Serra, P. Tilli, 
Dimensional crossover with a continuum of critical exponents for NLS on doubly periodic metric graphs, {\em Analysis and PDEs}, to appear.

\bibitem{AST-CVPDE}
R. Adami, E. Serra, P. Tilli,
NLS ground states on graphs,
\emph{Calc. Var. Partial Differential Equations} {\bf 54} (2015), no. 1, 743--761.

\bibitem{AST-JFA}
R. Adami, E. Serra, P. Tilli,
Threshold phenomena and existence results for NLS ground states on metric graphs,
\emph{J. Funct. Anal.} {\bf 271} (2016), no. 1, 201--223.  

\bibitem{AST-CMP}
R. Adami, E. Serra, P. Tilli, 
Negative energy ground states for the L2-critical NLSE on metric graphs,
\emph{Comm. Math. Phys.} {\bf 352} (2017), no. 1, 387--406.

\bibitem{BK}
G. Berkolaiko, P. Kuchment,
\emph{Introduction to quantum graphs},
Mathematical Surveys and Monographs 186, American Mathematical Society, Providence, RI, 2013.

\bibitem{BCT-p}
W. Borrelli, R. Carlone, L. Tentarelli,
Nonlinear Dirac equation on graphs with localized nonlinearities: bound states and nonrelativistic limit,
arXiv:1807.06937 [math.AP] (2018).

\bibitem{CDS-p}
C. Cacciapuoti, S. Dovetta, E. Serra,
Variational and stability properties of constant solutions to the NLS equation on compact metric graphs,
{\em Milan Journal of Mathematics}, to appear.

\bibitem{CFN-N}
C. Cacciapuoti, D. Finco, D. Noja,
Ground state and orbital stability for the NLS equation on a general starlike graph with potentials,
\emph{Nonlinearity} {\bf 30} (2017), no. 8, 3271--3303.

\bibitem{C}
T. Cazenave,
\emph{Semilinear Schr\"odinger equations},
Courant Lecture Notes in Mathematics 10, American Mathematical Society, Providence, RI, 2003.

\bibitem{DELL-JLMS}
J. Dolbeault, M.J. Esteban, A. Laptev, L. Michael,
One-dimensional Gagliardo-Nirenberg-Sobolev inequalities: remarks on duality and flows,
\emph{J. Lond. Math. Soc. (2)} {\bf 90} (2014), no. 2, 525-550.

\bibitem{D-JDE}
S. Dovetta,
Existence of infinitely many stationary solutions of the $L^2$-subcritical and critical NLSE on compact metric graphs,
\emph{J. Differential Equations} {\bf 264} (2018), no. 7, 4806--4821.
 
\bibitem{DT-p}
S. Dovetta, L. Tentarelli,
Ground states of the $L^2$-critical NLS equation with localized nonlinearity on a tadpole graph,
arXiv:1804.11107 [math.AP] (2018).

\bibitem{D2-p}
A. Duca,
Global exact controllability of the bilinear Schr\"odinger potential type models on quantum graphs,
arXiv:1710.06022 [math.OC] (2017).

\bibitem{KP-JDE}
A. Kairzhan, D.E. Pelinovsky,
Nonlinear instability of half-solitons on star graphs,
\emph{J. Differential Equations} {\bf 264} (2018), no. 12, 7357--7383.

\bibitem{LLS-JMAA}
Y. Li, F. Li, J. Shi,
Ground states of nonlinear Schr\"odinger equation on star metric graphs,
\emph{J. Math. Anal. Appl.} {\bf 459} (2018), no. 2, 661--685.

\bibitem{MP-AMRX}
J. L. Marzuola, D.E. Pelinovsky,
Ground state on the dumbbell graph,
arXiv:1509.04721 [math.AP] (2017).

\bibitem{MNS-APDE}
D. Mugnolo, D. Noja, C. Seifert,
Airy-type evolution equations on star graphs,
\emph{Anal. PDE} {\bf 11} (2018), no. 7, 1625--1652.

\bibitem{ST-JDE}
E. Serra, L. Tentarelli, 
Bound states of the NLS equation on metric graphs with localized nonlinearities,
\emph{J. Differential Equations} {\bf 260} (2016), no. 7, 5627--5644.

\bibitem{ST-NA}
E. Serra, L. Tentarelli,
On the lack of bound states for certain NLS equations on metric graphs,
\emph{Nonlinear Anal.} {\bf 145} (2016), 68--82.

\bibitem{T-JMAA}
L. Tentarelli, NLS ground states on metric graphs with localized nonlinearities,
\emph{J. Math. Anal. Appl.} {\bf 433} (2016), no. 1, 291--304.

\end{thebibliography}
\end{document}